\documentclass{amsart}
\usepackage{amsmath}
\usepackage{amssymb}
\usepackage{mathrsfs}
\usepackage{stmaryrd}
\usepackage[all]{xy}
\usepackage{amsfonts}

\makeatletter \renewcommand{\everyentry@}{\vphantom{A_{[]}�{[]}}}
\makeatother

\newtheorem{theorem}{Theorem}[subsection]

\newtheorem{lemma}[theorem]{Lemma}
\newtheorem{prop}[theorem]{Proposition}
\newtheorem{co}[theorem]{Corollary}

\theoremstyle{definition}

\newtheorem{example}[theorem]{Example}

\newtheorem{question}[theorem]{Question}
\theoremstyle{remark}
\newtheorem{remark}[theorem]{Remark}
\newtheorem{Notation}[theorem]{Notations}
\numberwithin{equation}{subsection}

 \def \E{\mathcal E}

\def \Z {\mathbb Z}
\def \inj {\hookrightarrow }
\def \to {\rightarrow}
\def \spec \text{spec}
 \def \M{\mathfrak M}

\def \e { {\underline \epsilon}}

\def \p {\underline \pi}

\def \GL {\t{GL}}

\def \Hom {\textnormal{Hom}}

\def  \cris {\textnormal{cris}}
\DeclareMathOperator{\gal}{Gal}

\def \Q {\mathbb Q}
\def \C {\mathbb C}
\def \t {\textnormal}
\def \Z {\mathbb Z}

\def \ito {\overset  \sim  \to}
\def \O {\mathcal O}

\def \gs {\mathfrak S}
\def \gu {\mathfrak u}
\def \ur {\t{ur}}
\def \D {\mathcal D}
\def \gf {{\mathfrak f}}
\def \N {\mathfrak N}

\def \dR {{\textnormal{dR}}}

\def \v {\vee}

\def \Fil {\t{Fil}}

\def \gt {\mathfrak t}
\def \CM {\mathcal M}
\def \CL {\mathcal L}
\def \CN {\mathcal N}

\def \acris {{A_{\t{cris}}}}

\def \dr {\textnormal{dR}}

\def \R {\mathcal R}

\def \st {\textnormal{st}}

\def \upi {\underline \pi }

\def \< {\left <}
\def \> {\right >}

\def \hR {{\widehat \R} }
\def \hM {{\hat \M}}

\def \upi {{\underline{\pi}}}
\def \unu {{\underline{\nu}}}

\def \inv {{\t{inv}}}
\def \rig {{ {\t{B}^+_\t{rig}}}}

\def \gv {{\mathfrak v}}
\def \fe {{\mathfrak e}}
\begin{document}

\title{Compatibility of Kisin modules for different uniformizers}

\author{Tong Liu}
\address{Department of Mathematics, Purdue University, Indiana, 47907, USA.}
\email{tongliu@math.purdue.edu}






\begin{abstract}
Let $p$ be a prime and $T$ a lattice inside a semi-stable representation $V$. We prove that Kisin modules associated to $T$   by selecting different uniformizers are isomorphic after tensoring a  subring in $W(R)$. As  consequences, we show that  several lattices inside the filtered $(\varphi, N)$-module of $V$ constructed from Kisin modules are  independent on the choice of uniformizers. Finally we use a similar strategy to show that the Wach module can be recovered from the $(\varphi, \hat G)$-module associated to $T$ when $V$ is crystalline and the base field is unramified.

\end{abstract}

\maketitle


\footnote{The author is partially supported by NSF grant DMS-0901360.}
\tableofcontents
\section{Introduction}
Let $k$ be a perfect field of characteristic $p$, $W(k)$ its ring of Witt vectors,
$K_0 = W(k)[1/p]$, $K/K_0$ a finite totally ramified extension,  $G_K := \gal(\overline K /K)$.

To understand the $p$-adic Hodge structure of $G_K$-stable $\Z_p$-lattices in semi-stable representations, the method of Kisin modules is powerful. Recall the definition of Kisin modules in the following: We  \emph{fix} a  uniformiser $\pi\in K$ with Eisenstein polynomial $E(u)$. Put $\gs := W(k)\llbracket u\rrbracket$. $\gs$ is equipped with a Frobenius endomorphism $\varphi$ via $u\mapsto u^p$ and the natural Frobenius on $W(k)$. A \emph{Kisin module of height $r$} is a finite free $\gs$-module $\M $ with $\varphi$-semi-linear endomorphism $\varphi_\M : \M \to \M $ such that $E(u)^r \M \subset \langle\varphi_\M (\M)\rangle $, where $\langle \varphi_\M  (\M)\rangle$ is the $\gs $-submodule of $\M$ generated by $\varphi_\M(\M )$. By the result of Kisin \cite{kisin2},  for any $G_K$-stable $\Z_p$-lattice $T$ inside a semi-stable representation $V$ with Hodge-Tate weights in $\{0, \dots, r\}$, there exists a unique Kisin module $\M (T)$ of height $r$ attached to $T$ (see \S 2.1 for more precise meaning of this sentence).

It is obvious that the construction of Kisin modules depends on the choice of unformizer $\pi$. If we choose another uniformizer $\pi'$ of $K$ then we get another $\M' (T)$. A natural question is: what is the relationship between $\M(T)$ and $\M'(T)$?

It turns out that each choice of uniformizer $\pi$ determines an embedding $\gs\inj W(R)$ via $u \mapsto [\underline \pi]$ (see \S2.1 for details of the definition of $W(R)$ and $[\upi]$). We denote $\gs_\upi$ and $\gs_{\upi'}$ for the image of embedding determined by $\upi$ and $\upi'$. By the main result of \cite{liu4}, there exists a $G_K$-action on $W(R) \otimes_{\varphi, \gs_\upi} \M (T)$ which commutes with $\varphi_\M$. In this paper, we prove the following:

\begin{theorem} \label{main1}There exists a $W(R)$-linear isomorphism $$W(R) \otimes_{\varphi, \gs_\upi}\M(T) \simeq W(R) \otimes_{\varphi, \gs_{\upi'}}\M'(T)$$ compatible with $\varphi$-actions and $G_K$-actions on the both sides.
\end{theorem}

In fact,  $W(R)$ in the above isomorphism can be replaced by much smaller ring $\tilde \gs _{\upi, \upi'}$, and even smaller ring $\gs _{\upi, \upi'}$ when $V$ is crystalline. See Theorem \ref{thm: main1.5} for more details.
It turns out that we can extend Theorem \ref{main1} to discuss the relation between Kisin modules and Wach modules (\cite{Ber}). Assume that $K= K_0$ is unramified and let $T$ be a $G_K$-stable $\Z_p$-lattice inside a crystalline representation. Then we can attach Wach module $\N(T)$ and Kisin module $\M(T)$ to $T$.  Let $\zeta_{p^n}$ be a primitive $p^n$-th root of unity.  Set $K_{p^\infty}:= \bigcup^\infty_{n =1}K(\zeta_{p^n})$ and $H_{p ^\infty}:= \gal (\overline K / K_{p^\infty})$. The following Theorem describes a direct  relation between the Kisin module and the Wach module.
\begin{theorem}\label{Thm: Wach&Kisin}

$\N(T) \simeq (\hR \otimes_{\varphi, \gs} \M(T)) ^{H_{p^\infty}}. $
\end{theorem}
Here $\hR\subset W(R)$ is a subring constructed in \S2.2, \cite{liu4} and $\hR \otimes_{\varphi, \gs } \M \subset W(R) \otimes_{\varphi, \gs}\M$ was proved to be $G_K$-stable \emph{loc. cit.} (also see \S \ref{subsec:review}).

Theorem \ref{main1}  can be used to understand lattices in the filtered $(\varphi, N)$-modules attached to semi-stable representations of $G_K$.
More precisely,  Let $V$ be a de Rham representation of $G_K$ and $T \subset V$ a $G_K$-stable $\Z_p$-lattice. It is well known (\cite{Ber0}) that $V$ is semi-stable over a finite extension $K' /K$. By using the Kisin module attached to $T|_{G_{K'}}$, we can construct various lattices in either $D_{\st, K'} (V):= ( V ^\v \otimes_{\Q_p} B_\st)^{G_{K'}}$ or $D_\dR (V):= (V^\v \otimes_{\Q_p} B_\dR )^{G_K}$, where $V^\v$ denotes the dual of $V$. One consequence of Theorem \ref{main1} is that the constructions of  such  lattices are independent on the choice of $\pi$. In the end, we also discuss several lattices (inside the filtered $(\varphi, N)$-module) whose constructions are independent on Kisin's theory. But they are useful to discuss the $p$-adic Hodge properties for ($p$-adic completion of) the direct limit of de Rham representations. In particular, we hope these will be  useful to understand those representations discussed in \cite{Emerton1} and \cite{Cais}.

The arrangement of this paper  is as follows: In \S 2, we setup notations and summarize the facts needed for the proof of Theorem \ref{main1} and give a more precise version of Theorem \ref{main1}. We give the proof of Theorem \ref{main1} and Theorem \ref{Thm: Wach&Kisin} in \S 3. We also show that the compatibility of Kisin modules when base changes (see Theorem \ref{thm: main1.6}).  \S4 are devoted to  discuss various lattices inside filtered $(\varphi, N)$-modules attached to potentially semi-stable representations. We show that two types of lattices constructed from Kisin's theory do not depend on the choice of uniformizers and they are compatible with base change. In \S4.3, we show that several lattices (constructed without using Kisin' theory) may help us to understand the $p$-adic completion of direct limit of de Rham representations. In particular, we hope that our strategy is useful to those representations studied in \cite{Emerton1}. The last section is the Errata of \cite{liu-wd}.

{\bf Acknowledgement:} The author would like to thank  Brian Conrad for raising  this question  and Bryden Cais for very useful comments.

\section{Preliminary and main results}
\subsection{Kisin modules and $(\varphi, \hat G)$-modules}\label{subsec:review}
We set up notations and recall some facts on (integral) $p$-adic Hodge theory in this section. We fix a nonnegative integer $r$ throughout the paper.  Let $V$ be a semi-stable representation of $G_K$ with Hodge-Tate weights in $\{0, \dots, r\}$. Write $V^\v$ the dual of $V$.  By the well-known theorem of Fontaine and Colmez, the functor $$V \mapsto D_\st (V): = (V^\v \otimes_{\Q_p} B_\st) ^{G_K}$$ induces an \emph{anti-equivalence} between the category of semi-stable representations with Hodge-Tate weights in $\{ 0, \dots ,r\}$ and the category of weakly admissible filtered $(\varphi, N)$-modules $(D, \varphi, N, \{\Fil^i D_K\})$ with $\Fil^0D_K = D_K$ and $\Fil ^{r+1} D_K= \{0\}$. Here $D_K := K \otimes_{K_0}D$ as usual.
The readers should be careful that we use the \emph{contravariant} version of $D_\st$, which were denoted by $D^*_\st$ in many papers. But the current version of $D_\st$ is more convenient for integral theory.

Let $R= \varprojlim \O_{\overline K }/p$ where the transition maps are
given by Frobenius.  By the universal property of  the Witt vectors
$W(R)$ of $R$, there is a unique surjective projection map $\theta :
W(R) \to \widehat \O_{\overline K}$ to the $p$-adic completion $ \widehat \O_{\overline K}$ of
$\O_{\overline K}$, which lifts the  projection $R \to \O_{\overline K}/ p$
onto the first factor in the inverse limit. We denote by $A_{\t {cris}}$ the $p$-adic completion
of the divided power envelope of $W(R)$ with respect to
$\t{Ker}(\theta)$.
 As usual, we write $B_\t{cris}^+=
A_\t{cris}[1/p]$ and   $B_\t{dR}^+$ the $\t{Ker}(\theta)$-adic completion of $W(R)[1/p]$. For any subring $A \subset B^+_\t{dR}$, we define filtration on $A$ by  $\t{Fil}^i A = A \cap (\t{Ker}(\theta))^iB^+_\t{dR}$.


Now select a uniformizer $\pi$   of $K$. Let $E(u)\in W(k)[u]$  be the Eisenstein polynomial of $\pi$.
Let $\pi_n\in \overline K $ be a $p^n$-th root of $\pi$, such that
$(\pi_{n+1})^p=\pi_n$; write $\underline \pi=(\pi_n)_{n\geq 0}\in R$
and let $[\underline \pi ]\in W(R)$ be the Techm\"uller
representative. We embed the $W(k)$-algebra $W(k)[u]$ into
$W(R)\subset\acris$ by the map $u\mapsto [\underline \pi]$.  Recall $\gs= W(k)[\![u]\!]$. This
embedding extends to the  embedding $\gs \inj W(R)$ which  are  compatible with
Frobenious endomorphisms.

We denote  by $S$  the $p$-adic completion of the divided power
envelope of $W(k)[u]$ with respect to the ideal generated by $E(u)$. Write $S_{K_0}:= S[\frac{1}{p}]$.
There is a unique map (Frobenius) $\varphi_S: S \to S$ which extends
the Frobenius on $\gs$. We write $N_S$ for the $K_0$-linear derivation on $S_{K_{0}}$ such that
$N_S(u)= -u$. Let  $\Fil ^n S\subset S $ be the $p$-adic completion of the ideal generated by $\gamma_i (E(u)):= \frac{E(u)^i}{i!}$ with $ i \geq n$. One can show that the embedding $W(k)[u] \to W(R)$ via $u \mapsto [\upi]$ extends to the embedding $S \inj A_\cris$ compatible with Frobenius $\varphi$ and filtration (note that $E([\upi])$ is a generator of $\Fil ^1 W(R)$). We set $B^+_\st : = B^+_\cris [\gu]\subset B^+_\dR$ with
$\gu:=\log( [\p])$.

Let $K_\infty := \bigcup\limits_{n=0}^\infty K(\pi_n)$ and $\hat K$ its Galois closure over $K$.
Then $\hat K =\bigcup \limits_{n=1}^\infty K_\infty (\zeta_
{p^n})$ with $\zeta _{p^n} $ a primitive $p^n$-th root of unity.
Write $G_\infty:= \gal (\overline K / K_\infty)$,  $K_{p^\infty}=  \bigcup \limits _{n=1}^\infty
K(\zeta_{p^n})$,
$G_{p^\infty} := \gal(\hat K/ K_{p^\infty})$, $H_{K}:=
\gal (\hat K/ K_\infty)$ and $\hat G: =\gal (\hat K/K) $. For any $ g \in G_K$, $\e (g):= \frac{g(\upi)}{\upi}$ is a cocycle with value in $R$.
 Set  $\e:= (\zeta_{p^i}) _{i \geq 0} \in R$ and $t:= -\log([\e])\in \acris$ as usual.

As a subring of $\acris$, $S$ is not stable under the action of $G_K$,
though $S$ is fixed by $G_\infty$. Define a subring inside $B^+_\cris$:

$$\R_{K_0}: =\left\{x = \sum_{i=0 }^\infty f_i t^{\{i\}}, f_i \in S_{K_0} \t{
and } f_i \to 0\t{ as }i \to +\infty \right\}, $$ where $t ^{\{i\}}= \frac{t^i}{p^{\tilde q(i)}\tilde q(i)!}$ and $\tilde q(i)$ satisfies $i = \tilde q(i)(p-1) +r(i)$ with $0 \leq r(i )< p-1$. Define $\hR := W(R)\cap \R_{K_0}$. One
can show that $\R_{K_0}$ and $\hR$ are stable under the $G_K$-action and the $G_K$-action factors through  $\hat G$ (see \cite{liu4} \S2.2). Let $I_+R$ be the maximal ideal of $R$ and $I_+ \hR = W(I_+R) \cap \hR$. By Lemma 2.2.1 in \cite{liu4}, one have
$\hR / I_+\hR\simeq \gs/ u \gs  = W(k)$.

Recall that a \emph{Kisin module of height $r$} is a finite free $\gs$-module $\M $ with $\varphi$-semi-linear endomorphism $\varphi_\M : \M \to \M $ such that $E(u)^r \M \subset \langle\varphi_\M (\M)\rangle $, where $\langle \varphi_\M (\M)\rangle$ is the $\gs $-submodule of $\M$ generated by $\varphi_\M(\M )$. A morphism between two Kisin modules is just an  $\gs$-linear map compatible with Frobenius. As a subring of $\acris$ via $u \to [\upi]$, $\gs$ and  $ S$ are  not stable under the action of $G_K$, but stable under $G_\infty$. This allows us to define a functor $T_\gs$ from the category of Kisin modules to the category of finite free $\Z_p$-representations of $G_\infty$ via the following formula:
$$T_\gs (\M) : = \Hom_{\gs, \varphi} (\M , W(R)).$$
See \S2.2 in \cite{liu2} for more details of $T_\gs$. In particular, by Proposition 2.2.1 $loc. cit.$, we can change $\gs^\ur$ to $W(R)$ in the definition of $T_\gs$.

Let us review the theory of $(\varphi, \hat G)$-modules, which is a variation of that of Kisin modules.  Following \cite{liu4}, a \emph{finite free $(\varphi, \hat
G)$-module of height $ r$} is a triple $(\M , \varphi, \hat G)$ where
\begin{enumerate}
\item $(\M, \varphi_\M)$ is a finite free Kisin module of height $ r$;
\item $\hat G$ is a $\hR$-semi-linear $\hat G$-action on $\hat \M: =\hR
\otimes_{\varphi, \gs} \M$;
\item $\hat G$ commutes with $\varphi_{\hM}$ on $\hM$, \emph{i.e.,} for
any $g \in \hat G$, $g \varphi_{\hM} = \varphi_{\hM} g$;
\item regard $\M$ as a $\varphi(\gs)$-submodule in $ \hM $, then $\M
\subset \hM ^{H_{K}}$;
\item $\hat G$ acts on $W(k)$-module $M:= \hM/I_+\hR\hM\simeq \M/u\M$ trivially.
\end{enumerate}

 A morphism between two finite  free $(\varphi, \hat G)$-modules is a morphism of Kisin modules that  commutes with $\hat G$-action  on $\hM$'s.
  For a finite free  $(\varphi, \hat G)$-module  $\hM=(\M, \varphi, \hat G)$, 
  we can associate a $\Z_p[G_K]$-module:
\begin{equation}\label{hatT}
\hat T (\hM) := \t{Hom}_{\hR, \varphi}(\hR \otimes_{\varphi, \gs}
\M, W(R)),
\end{equation}
where $G_K$ acts on $\hat T(\hM)$ via $g (f)(x) = g (f(g^{-1}(x)))$
for any $g \in G_K$ and $f \in \hat T(\hM)$.

By Example 2.3.5 in
\cite{liu2}, there exists an element $\gt \in W(R)$ such that $\gt \mod p \not = 0$,
$\varphi (\gt)= c^{-1}_0E(u) \gt$, where $c_0p$ is the constant term of
$E(u)$. Such $\gt$ is unique up to $\Z_p ^\times$.  The following theorem summarizes main results in \cite{liu4}.

\begin{theorem}[\cite{liu4}]\label{review}
\begin{enumerate}
\item  $\hat T$ induces an anti-equivalence between the category of finite free
$(\varphi, \hat G)$-modules of height $r$ and the category of
$G_K$-stable $\Z_p$-lattices in semi-stable representations of $G_K$
with Hodge-Tate weights in $\{0, \dots, r\}$.
\item $\hat T $ induces a natural $W(R)$-linear injection
\begin{equation}\label{Eq: iota}
\hat \iota: \  W(R)\otimes_{\varphi, \gs} \M  \longrightarrow \hat T^\v
(\hat \M) \otimes_{\Z_p} W(R),
\end{equation}
 such that $\hat \iota$ is  compatible with Frobenius and $G_K$-actions on both sides. Moreover, $(\varphi(\gt))^r (\hat T^\v(\hM) \otimes
_{\Z_p}W(R)) \subset\hat \iota( W(R)\otimes_{\varphi, \gs} \M).$
\item There exists a natural isomorphism $T_\gs(\M) \ito \hat T (\hM)$ of $\Z_p[G_\infty]$-modules.
\end{enumerate}
\end{theorem}

\subsection{A refinement of Theorem \ref{main1}} Obviously, the theory described by Theorem \ref{review} depends on the choice of uniformizer $\pi$ in $K$. Fix a $G_K$-stable $\Z_p$-lattice $T$ inside a semi-stable representation $V$, if we select another uniformizer $\pi'$ then we obtain  $\M'$ and $\hat \iota' $ in Equation \eqref{Eq: iota}. As indicated in the introduction, one main goal of this paper is to understand the relation between $\M$ and $\M'$. Let $\gs_{\upi}$ (resp. $S_{\upi}$) denote the image of embedding $ \gs \inj W(R)$ (resp. $S \inj A_\cris$) via $u \mapsto [\upi]$. Write $\upi' = \unu \upi$ with $\unu = (\nu_n)_{n \geq 0} \in R$. Note that $\nu_0$ is a unit. So $\log ([\unu]) \in B^+_\cris$.

We denote by $\gs_{\upi'}$ and $S_{\upi'}$ the subrings of $W(R)$ and $A_\cris$ respectively via $u \mapsto [\upi']$. Let $\tilde S_{\upi, \upi'}$ be the smallest ring inside $B^+_\cris$ containing $S_{\upi}[1/p]$, $S_{\upi'}[1/p]$ and $\log([\unu])$. Set $\tilde \gs_{\upi, \upi'}: = W(R) \cap \tilde  S_{\upi, \upi'}$. Similarly,  let $ S_{\upi, \upi'}$ be the smallest ring inside $B^+_\cris$ containing $S_{\upi}[1/p]$ and $S_{\upi'}[1/p]$ and  set $ \gs_{\upi, \upi'}: = W(R) \cap  S_{\upi, \upi'}$.

\begin{theorem}\label{thm: main1.5} Notations as the above, we have
$$\hat \iota ( \tilde \gs_{\upi, \upi'} \otimes _{\varphi, \gs_{\upi}} \M) = \hat \iota' (\tilde \gs_{\upi, \upi'} \otimes _{\varphi, \gs_{\upi'}} \M')$$
as submodules of $T^\v \otimes_{\Z_p} W(R)$.

If $V$ is crystalline then $\tilde \gs_{\upi, \upi'}$ in the above equation can be replaced by $\gs_{\upi, \upi'}$.
\end{theorem}
\begin{remark}\begin{enumerate}\item Let $\nu= \pi/\pi'$. If $\nu  \in W(k)^\times $ then we can arrange $\pi'_n$ so that $[\upi]= [\bar\nu][\upi']$ with $\bar \nu = \nu\mod p \in k^\times$.  Hence $\gs_{\upi, \upi'}= \gs_{\upi} = \gs_{\upi'}$.

\item If $\nu  \not \in W(k)^\times $ then the situation could be  more complicated. So far we do not have a good description for $\tilde \gs _{\upi, \upi'}$,  even for $\gs _{\upi, \upi'}$ . We warn the readers that $\gs_{\upi, \upi'}$ may be larger than the smallest ring containing $\gs_\upi$ and $\gs_{\upi'}$. For example, let $\tilde E (u)$ be the  Eisenstein polynomial of $\pi'$. Then $\tilde E([\upi'])/  E([\upi])$ is a unit in $W(R)$,  because $\Fil^1 W(R)$ is a principal ideal and $E([\upi])$ and $\tilde E([\upi'])$ are generators of $\Fil^1 W(R)$. Hence $ x= \varphi ( \tilde E ([\upi'])/  E([\upi]))= \tilde E([\upi']^p)/  E([\upi]^p) \in W(R)$. But $\frac{ E([\upi]^p)}{p}$ is a unit in $S_{\upi}$. Therefore, $x \in \gs_{\upi, \upi'}$. In general, $x$ is not in the smallest ring containing $\gs_{\upi}$ and $\gs_{\upi'}$. See the following example.

\end{enumerate}
\end{remark}
\begin{example}\label{ex:1} Let $K= \Q_p(\zeta_p)$. Let $\pi= \zeta_p -1$ and $\pi' = \zeta_p \pi$. We can choose $\upi$ and $\upi'$ such that $\upi' =\upi \e'  $ with ${\e'}^p = \e$. Then the smallest ring $\tilde \gs$ containing $\gs_\upi$ and $\gs_\upi'$ is inside $W(k)[\![ [\upi], [\e']-1]\!]$. If $x $ is in $\tilde \gs$ then  $\varphi (x)$ is in $W(k)[\![ [\upi], [\e]-1]\!] \subset \R_{K_0}\cap W(R) = \hR$. On the other hand,  since $\frac{\varphi^2( E([\upi]))}{p}$ is a unit in $S_\upi$,  we can write $\varphi (x)$ as a series in $K_0[\![[\pi], [\e]-1]\!]$. It is easy to see that this series is not in $W(k)[\![ [\upi], [\e]-1]\!]$. But by Lemma 7.1.2, \cite{liu2}, for any $y \in \R_{K_0}$, there is only one way to expand $y$ in a series in $K_0[\![[\upi], [\e]-1]\!]$. So $\varphi (x)$ is not in $W(k)[\![ [\upi], [\e]-1]\!]$. Contradiction and $\varphi(x) \not \in \tilde \gs$.

\end{example}
\begin{Notation}\label{convention} 
We will reserve $\varphi$ and $N$ to denote Frobenius action, monodromy action on many different rings and modules. To distinguish them,
we sometime add subscripts to indicate over which those structures are defined. For example, $\varphi_\M$ is the Frobenius defined on $\M$. We always drop these subscripts if no confusions arise. As we have indicated as before, Kisin's theory (and its related theory, like the theories of Breuil modules and $(\varphi, \hat G)$-modules, which will be used below) depends on the choice of the unformizer $\pi$, or more precisely, depends on the choice of $\pi_n$ and hence the embedding $\gs \inj W(R)$ via $u \mapsto [\upi]$. We add subscripts $\upi$ to subrings in $W(R)$ to denote subrings  (like $\gs, S$) whose embeddings to $B^+_\dR$ depends on the embedding  $\gs \inj W(R)$ via $u \mapsto [\upi]$. But we always drop subscripts when we just discuss the general theory where  the embedding $\gs \inj W(R)$ via $u \mapsto [\upi]$ is always fixed.
 Finally, $\gamma_i(x)$, $\t{M}_{d \times d}(A)$ and  $\t{Id}$ denote the standard divided power $\frac{x^i}{i!}$,
 the ring of $d \times d$-matrices with coefficients in ring $A$ and the identity map respectively; $V^\v$ denotes the dual of a representation $V$.
\end{Notation}
\subsection{Some facts on the theory of Breuil modules} We will use extensively the theory of {Breuil modules}, which we review in this subsection.  Following \cite{b2}, a \emph{filtered $\varphi$-module over
$S[\frac 1 p ]$} is a finite free $S[\frac 1 p]$-module $\D$ with
 \begin{enumerate}
  \item a $\varphi_{S}$-semi-linear morphism $\varphi_\D: \D \to
        \D$ such that the determinant of $\varphi_\D$ is invertible in
        $S[\frac 1 p ]$,
  \item a decreasing filtration over $\D$ of $S_{K_0}$-modules
        $\{\Fil^i (\D)\}_{i \in \Z}$ with $\Fil^0 (\D)= \D $  and
        $\Fil^iS_{K_0}\cdot\Fil^j(\D)\subset \Fil^{i+j}(\D). $
 \end{enumerate}
Similarly, we define \emph{filtered $\varphi$-modules over $S$} by changing $S[\frac 1 p]$ to $S$ everywhere in the above definition, but we still require that the determinant of $\varphi$ is in $S[\frac 1 p]$.

  A \emph{Breuil module} is a filtered $\varphi$-module $\D$ over $S[\frac 1 p ]$
with following extra monodromy structure: a $K_0$-linear map (monodromy)  $N_\D: \D \to \D$ such
that
\begin{enumerate}
\item for all $f\in S_{K_0}$ and $m \in \D$, $N_\D (fm)= N_S(f)m + fN_\D(m)$,
\item $N_\D\varphi = p \varphi N_\D$,
\item $N_\D(\Fil^i\D)\subset \Fil^{i-1}\D.$
\end{enumerate}

A filtered $(\varphi, N)$-module $D$ is called \emph{positive} if $\Fil ^0 D_K = D_K$.
It turns out that the category of positive filtered $(\varphi, N)$-modules and the category of Breuil modules are equivalent. More precisely, for any positive filtered $(\varphi, N)$-module $(D, \varphi, N , \Fil^i D_K)$,
we can associate a Breuil module $\D$ by defining $\D = S
\otimes_{W(k)}D $;
 $\varphi_\D : = \varphi_S \otimes \varphi_D$;  $N_\D := N_S \otimes\t{Id} + \t{Id}\otimes N_D;  $
Define $\Fil^0 \D:=\D$ and by induction
$$\Fil^{i+1}\D:= \{ x\in \D \mid N(x) \in \Fil^i \D \t{ and } f_\pi (x) \in \Fil^{i+1}D_K\}, $$
where $f_\pi : \D \twoheadrightarrow D_K$ is defined by $s(u)\otimes
x \mapsto s(\pi)x$.

In \S 6 of \cite{b2}, Breuil proved the above
functor $\D: D \to S \otimes_{W(k)}D$ is an equivalence of
categories. Furthermore, $D$ and $\D(D)$ give rise to the same Galois
representations (Proposition 4.1.1.2 in \cite{b3}), namely, there is a natural isomorphism $$\Hom_{ W(k), \varphi, N, \Fil^i } (D, B^+_\st)\simeq \Hom_{S, \varphi, N, \Fil ^i } (\D(D), \widehat B^+_\st) $$ as $\Q_p[G_K]$-modules. Here $\widehat B^+_\st$ is the period ring defined in \cite{b0}.
\begin{remark}\label{remark set N}
In the theory of Kisin and Breuil modules, we use implicitly or explicitly the above isomorphism to connect   Galois representations associated  to  filtered $(\varphi, N)$-modules with those of  Breuil modules or Kisin modules. To make the above isomorphism, one set the monodromy $N$ on $B^+_\st$ via $N(\gu) =1$ (see \S3.1.1 in \cite{b3}). So strictly speaking, the monodromy structure on $B^+_\st$ may depend on the choice of uniformizer  $\pi$. On the other hand, pick another uniformizer  $\pi'$ of $K$. We have $\pi= \nu\pi' $ with $\nu$ a unit in $\O_K$. Hence $\gu = \gu ' + \beta$ with  $\gu'= \log([\upi'])$ and $\beta$ in $B^+_\cris$. So $N (\gu') =1$ if and only $N(\gu)=1$. This shows that the monodromy structure on $B^+_\st$ is unique when we declare $N(\gu)=1$ and it does not depend on the choice of \emph{uniformizers} in $\O_K$.

\end{remark}

One can naturally extend Frobenius from $\D$ to $A_\cris \otimes _S \D$ via $\varphi := \varphi_{A_\cris} \otimes\varphi_ \D$.
 We define a semi-linear $G_K$-action on $A_\cris \otimes_S \D$ via
\begin{equation}\label{Eq: G_K-action}
    \sigma(a \otimes x  )= \sum_{i=0}^\infty \sigma(a)
    \gamma_i (-\log([\e(\sigma)])) \otimes N^i (x)
\end{equation}
for $\sigma \in G_K $,  $x\in \D$ and $a \in \acris$. This $G_K$-action commutes with $\varphi$ on $\acris \otimes_S \D$ (see Lemma 5.1.1 in \cite{liu3}).

Given a Kisin module $\M$, one can define a filtered $\varphi$-module $\CM_\gs (\M)$ over $S$ in the following: Set $\CM:= \CM_\gs (\M) = S \otimes_{\varphi, \gs} \M$ and extend Frobenius $\varphi_\M$ to $ \CM$ by $\varphi_\CM := \varphi_{S} \otimes  \varphi_\M$;  Define a filtration on $\CM$ via
\begin{equation}\label{Eq: filtration fron Kisin module}
\Fil ^i \CM:= \{x  \in \CM| 1 \otimes \varphi_\M (x) \in \Fil ^i S \otimes_\gs \M \},
\end{equation}
where $1 \otimes \varphi_\M : \CM= S \otimes_{\varphi, \gs}\M \to S \otimes_\gs \M $ is an $S$-linear map.

Now let $V$ be a semi-stable representation of $G_K$ with Hodge-Tate weights in $\{0, \dots , r \}$, $T \subset V$ a $G_K$-stable $\Z_p$-lattice inside $V$, $D= D_\st (V)$ the filtered $(\varphi, N)$-module attached to $V$ and $(\M, \varphi, \hat G )$ the $(\varphi, \hat G)$-module attached to $T$ via Theorem \ref{review}. Let $\D = \D (D)$ be the Breuil module and $\CM := \CM_\gs (\M)$. The following theorem summarize the relations between Breuil modules, filtered $(\varphi, N)$-modules and $(\varphi, \hat G)$-modules (Kisin modules):

\begin{theorem}\label{thm: big-review} Notations as the above. Then the following statements hold:
\begin{enumerate}
\item There exists a natural isomorphism $\alpha: \Q_p \otimes_{\Z_p} \CM_\gs (\M) \simeq \D$ as filtered $\varphi$-modules over $S[\frac 1 p ]$.

\item There exists a natural injection
 \begin{equation}\label{Eq: Qp comparison}
 \begin{split}
 \xymatrix{ \iota : \acris \otimes_S \D \ar[r] &     V^\v \otimes_{\Z_p} {\acris} }
 \end{split}
  \end{equation}
  which is compatible with Frobenius $\varphi$ and $G_K$-actions on the both sides.

\item The isomorphism $\alpha$ induces the following commutative diagram
 \begin{equation}\label{Dg: 1}
 \begin{split}
 \xymatrix{  \acris \otimes_S \D \ar[r]^-{\iota} &     V^\v \otimes_{\Z_p} {\acris}  \\ W(R) \otimes _{\varphi, \gs}\M \ar[r]^-{\hat \iota} \ar@{^(->}[u]& T^\v \otimes_{\Z_p} W(R) \ar@{^(->}[u]}
 \end{split}
  \end{equation}
  where the top map is Equation \eqref{Eq: Qp comparison} and the bottom map is Equation \eqref{Eq: iota}. The left vertical arrow is induced by $\alpha$ restricted to $\gs \otimes_{\varphi, \gs}\M$ and the right arrow is induced by the injection $T^\v \inj V^\v.$
\end{enumerate}
\end{theorem}
\begin{proof} (1) follows the compatibility between Kisin modules and Breuil modules. See \S3.4 in \cite{liu3}. (2) is proved in \S 5.2 in \cite{liu3}. The key point is that $\Hom_{\acris, \varphi, \Fil^i} (\acris \otimes_S \D , B^+_\cris)$ is canonically isomorphic to $V$ as $\Q_p [G_K]$-modules. The proof of (3) relies on the construction of $(\varphi, \hat G)$-modules. See Theorem 5.4.2 in \cite{liu2} and Proposition 3.1.3 in \cite{liu4}.
\end{proof}
\section{The proof of the main theorems}
We will prove Theorem \ref{thm: main1.5}, Theorem \ref{thm: main1.6} and Theorem \ref{Thm: Wach&Kisin} in this section. Our strategy is almost the same as that in \S3.2 in \cite{liu4}.
\subsection{The proof of Theorem \ref{thm: main1.5}}\label{subsec: 3.1}


To prove Theorem \ref{thm: main1.5}, we first show that the injection $\iota$ in Equation \eqref{Eq: Qp comparison} does not depend on the choices of uniformizer. More precisely, let $\D'$ denote the Breuil module attached to $V$  and $\iota'$  the injection in Equation \eqref{Eq: Qp comparison} for the choice of uniformizer $\pi'$. We claim:

\begin{lemma}\label{lem:3.1.1} There exists an $\acris$-linear isomorphism $$\beta : \acris \otimes_{S_\upi} \D \to \acris  \otimes_{S_{\upi'}} \D'$$ which is compatible with $G_K$-actions and  Frobenius  such that the following diagram commutes
 \begin{equation}\label{Dg: 2}
 \begin{split}
 \xymatrix{  \acris \otimes_{S_\upi } \D \ar[r]^{\iota} &     V^\v \otimes_{\Z_p} {\acris}  \\ \acris \otimes _{S_{\upi'}}\D'  \ar[r]^{\iota'}  \ar[u]^\beta_\wr & V^\v \otimes_{\Z_p} \acris \ar@{=}[u]}
 \end{split}
  \end{equation}

\end{lemma}
\begin{proof} Let $I_+ S = S \cap u K_0 [\![u]\!]$ and $D:= \D/ I_+S \D$. Then $D$ is a finite dimensional $K_0$-vector space with Frobenius $\varphi$ and monodromy $N$ on $D$ induced from that on $\D$. Proposition 6.2.1.1 in \cite{b0} showed that
there exists a unique $(\varphi, N)$-equivariant section $s : D \inj \D$ and $\D = S \otimes_{W(k)} s(D) $ as $S$-modules. By Proposition 2.2.2 in \cite{liu-wd}, $s(D)\subset V^\v \otimes_{\Z_p}\acris \subset V^\v \otimes_{\Q_p} B^+_\st$ has the following relation with $D_\st (V) = (V^\v \otimes _{\Q_p} B_\st ^+)^{G_K}$: There exists a (necessarily unique) isomorphism $i: D_\st (V) \to D$ compatible with $\varphi$ and $N$ such that the following diagram commutes
\begin{equation}\label{i}
\begin{split}\xymatrix{D_\st (V) \ar[d]^-i_\wr  \ar@{^{(}->}[r] & V^\v \otimes _{\Q_p} B^+_\st\ar[d]^{\mod \gu} \\
s(D) \ar@{^{(}->}[r]^-{\iota} & V^\v \otimes _{\Q_p} B^+_\cris}
\end{split}
\end{equation}
where $\gu= \log ([\upi]) \in B^+_\st$,  and the inverse of $i$ is given by  $ y \mapsto \sum \limits_{n=0}^\infty N^n(y) \otimes \gamma _n  (\gu)$.
If we fix a $K_0$-basis $\tilde e_1 , \dots , \tilde e_d$ of $D_\st (V)$, then by the above diagram, we obtain a basis $e_1 , \dots , e_d$ of $s(D)$ by modulo $\gu$ to $\tilde e_1 , \dots , \tilde e_d$, and $$(e_1, \dots, e_d)= (\tilde e_1 , \dots, \tilde e_d)\sum \limits^\infty_{n= 0} \gamma_n (-\gu) (\bar N) ^n ,  $$ where $\bar N \in \t{M}_{d \times d}(K_0)$ is the matrix such that $N(\tilde e_1, \dots, \tilde e_d) = (\tilde e_1, \dots , \tilde e _d) \bar N$.

Now by changing to another uniformizer $\pi'$, we get  $s'(D')$ injects to $V^\v \otimes_{\Q_p} B^+_\cris$. Modulo $\tilde e_1 , \dots, \tilde e_d$ by $\gu' = \log([\upi'])$, we get the basis $e'_1 , \dots, e'_d$ of $s'(D')$ and
$$(e'_1, \dots, e'_d)= (\tilde e_1 , \dots, \tilde e_d)\sum \limits^\infty_{n= 0} \gamma_n (-\gu') (\bar N) ^n. $$
Write $\upi = \underline \nu \upi'$ with $\underline \nu = (\nu_n)_{n\geq 0 } \in R$. Since $\nu_0 $ is a unit in $\O_K$, $\log([\underline \nu])$ is in $B^+_\cris$. Now we get
\begin{equation}\label{Eq: difference of basis}
(e_1 ,\dots,  e_d)= (e'_1 , \dots, e'_d) \sum_{n=0}^\infty \gamma_n(-\log([\underline \nu])) (\bar N)^n
 \end{equation}
 We remark the sum in the right side of the above equation is indeed a finite sum because $\bar N ^n= 0$ if $n $ is large enough. Now the lemma follows the facts that $s(D) \otimes _{W(k)} S \simeq \D$ as $S$-modules and that the matrix $\sum \limits _{n=0}^\infty \gamma_n(-\log([\underline \nu])) (\bar N)^n   $ has coefficients in $B^+_\cris$.

\end{proof}

\begin{co} Let $\hat e_1 , \dots, \hat e_d $ be an $S_\upi[\frac 1 p ]$-basis of $\D$  and $\hat e'_1 , \dots, \hat e'_d $ an $S_{\upi'}[\frac 1 p ]$-basis of $\D'$. Then $(\hat e'_1 ,\dots  \hat e'_d ) = (\hat e_1 ,\dots,  \hat e_d ) X$ with an invertible matrix $X$ whose entries are in  $\tilde S_{\upi, \upi'}$. If $V$ is crystalline then $X$ has entries in $S_{\upi, \upi'}$.
\end{co}

Now we are ready to prove Theorem \ref{thm: main1.5}. Let $\hat e_1, \dots , \hat e_d$ be an $\gs$-basis of $\M$, and $\hat e'_1 , \dots , \hat e'_d $ an $\gs$-basis of $\M'$ respectively. Regarding $\M$ as an $\varphi (\gs)$-submodule of $\D$ via the isomorphism $\alpha : \Q_p \otimes_{\Z_p} \CM_\gs  (\M) \simeq \D$ by Theorem \ref{thm: big-review} (1), we can regard $\{ \hat e_i\}$ as an $S_\upi [\frac 1 p ]$-basis of $\D$. Similarly, $\{ \hat e'_i\}$ is an $S_\upi'[\frac 1 p]$ basis of $\D'$. So by the above corollary, we may write $(\hat e' _1, \dots, \hat e'_d )= (\hat e_1 , \dots , \hat e_d ) X$ with $X$ having entries in $\tilde S_{\upi , \upi' }$, and in $S_{\upi, \upi'}$ if $V$ is crystalline.

Now to prove Theorem \ref{thm: main1.5}, it suffices to show that $X$ has entries in $W(R)$. Define an ideal $$I^{[1]}W(R):= \{ x \in  W(R)| \varphi ^n (x) \in \Fil ^1 W(R) \t{ for  all } n \geq 0  \}.$$

By Proposition 5.1.3 in \cite{fo3}, $I^{[1]}W(R)$ is a principal ideal. We record the following useful lemma:
\begin{lemma} \label{lem: key} Let $a$ be a generator of $I^{[1]} W(R)$ and $x \in B^+_\cris$. If $ax \in W(R)$ then $x \in W(R)$
\end{lemma}
 \begin{proof} See the proof in Lemma 3.2.2 in \cite{liu4}. Note  that $\varphi (\gt)$ is also proved to be  a generater of $I^{[1]} W(R)$ there.  \end{proof}
 Note that the construction of $\gt$ also depends on the choice of $\upi$. So we denote $\gt'$ for the choice of $\upi'$. By Theorem \ref{review} (2), we have $\hat \iota (\hat e_i) \in T^\v \otimes _{\Z_p} W(R)$ and then $(\varphi (\gt' )) ^r \hat \iota (\hat e_i) $ is in $\hat \iota' (W(R) \otimes_{\varphi, \gs_{\upi'}} \M')$. Then Theorem \ref{thm: big-review} (2), (3) implies that  $ (\varphi (\gt'))^r  X $ has entries in $W(R)$. Then $X$ must has entries in $W(R)$ by the above lemma. This completes the proof of Theorem \ref{thm: main1.5}.
 \subsection{Compatibility of basis change} Assume that $T$ is a  $G_K$-stable $\Z_p$-lattice in semi-stable representation $V$ of $G_K$ with $(\M, \varphi, \hat G)$ the corresponding $(\varphi , \hat G)$-module via the fixed uniformizer $\pi$. Let $K'$ be a finite extension of $K$ and $(\M ', \varphi, \hat G)$  the $(\varphi, \hat G)$-module corresponding to $T|_{G_{K'}}$ via the fixed uniformizer $\pi'$ of $\O_{K'}$. We would like to compare $\M$ and $\M'$.

  Let $k'$ be the residue field of $\O_{K'}$ and $K'_0:= W(k ') [\frac 1 p]$.  Suppose that $\upi = \underline \nu {\upi'} ^m$ where  $\unu = (\nu_n)_{n \geq 0} \in R$ with $\nu_0 \in \O_{K'}^\times$ a unit. Set $ \tilde S_{\upi, \upi'}\subset B^+_\cris$ be the smallest $K_0$-algebra  containing $ S_\upi[\frac 1 p ]$, $S_{\upi' } [\frac 1 p]$ and $\log([\unu])$, and $\tilde \gs_{\upi, \upi'} = W(R) \cap \tilde S_{\upi, \upi'}$. Let
 $  S_{\upi, \upi'}\subset B^+_\cris$ be the smallest $K_0$-algebra  containing $ S_\upi[\frac 1 p ]$, $S_{\upi' } [\frac 1 p]$, and $\gs_{\upi, \upi'} = W(R) \cap  S_{\upi, \upi'}$.  The following result is very similar to Theorem \ref{thm: main1.5}.
\begin{theorem}\label{thm: main1.6} Notations as the above, we have
$$\hat \iota ( \tilde \gs_{\upi, \upi'} \otimes _{\varphi, \gs_{\upi}} \M) = \hat \iota' (\tilde \gs_{\upi, \upi'} \otimes _{\varphi, \gs_{\upi'}} \M')$$
as submodules of $T^\v \otimes_{\Z_p} W(R)$.

If $V$ is crystalline then $\tilde \gs_{\upi, \upi'}$ in the above equation can be replaced by $\gs_{\upi, \upi'}$.
\end{theorem}
\begin{proof}Here we provide a similar proof to that of Theorem \ref{thm: main1.5}. We first reproduce Lemma \ref{lem:3.1.1}. We claim
 that there exists an $\acris$-linear isomorphism $$\beta : \acris \otimes_{S_\upi} \D \to \acris  \otimes_{S_{\upi'}} \D'$$ which is compatible with $G_{K'}$-action, Frobenius such that the following diagram commutes
$$
 \xymatrix{  \acris \otimes_{S_\upi } \D \ar[r]^{\iota} &     V^\v \otimes_{\Z_p} {\acris}  \\ \acris \otimes _{S_{\upi'}}\D'  \ar[r]^{\iota'}  \ar[u]^\beta_\wr & V^\v \otimes_{\Z_p} \acris \ar@{=}[u]}
 $$
The only difference  is that $\beta$ is only $G_{K'}$-equivariant.  To prove the claim, we use almost the same proof as that of Lemma \ref{lem:3.1.1} but with extra care on the monodromy structure of $B^+_\st$. Write $V': = V|_{G_{K'}}$ and $D':  = \D' / I_+ S \D'$. We still have Diagram \eqref{i} for $V'$ and $V$. Fix a $K_0$-basis $ \tilde e_1 , \dots , \tilde e_d$ of $D_\st(V)$. Then $\{\tilde e_i\}$ is a $K'_0$-basis of $D_\st (V')$. By modulo $\gu' = \log ([\upi'])$, we have a basis $e'_1 , \dots, e'_d $ of $s'(D')$ and the relation
$$(e'_1, \dots, e'_d)= (\tilde e_1 , \dots, \tilde e_d)\sum \limits^\infty_{n= 0} \gamma_n  (-\gu') (\bar N') ^n ,  $$ where $\bar N' \in \t{M}_{d \times d}(K'_0)$ is the matrix such that $N(\tilde e_1, \dots, \tilde e_d) = (\tilde e_1, \dots , \tilde e _d) \bar N'$. Note that we use the convention $N(\gu') =1$ by Remark \ref{remark set N}.

Similarly, we obtain a $K_0$-basis $e_1, \dots ,e_d$ of $s(D)$ and $$(e_1, \dots, e_d)= (\tilde e_1 , \dots, \tilde e_d)\sum \limits^\infty_{n= 0} \gamma_n  (-\gu) (\bar N) ^n ,  $$ with $\bar N \in \t{M}_{d \times d}(K_0)$  the matrix such that $N(\tilde e_1, \dots, \tilde e_d) = (\tilde e_1, \dots , \tilde e _d) \bar N$. But the convention used here is $N(\gu)=1$. To find the relation between $\bar N$ and $\bar N'$,  let us fix the convention $N(\gu) =1 $. Since $ \upi = \unu {\upi'} ^m$, we have $\gu = m \gu' + \log([\unu])$ and then $N(\gu') = \frac 1 m$. Consider the equation \[ (e'_1, \dots, e'_d)\sum \limits^\infty_{n= 0} \gamma_n  (\gu') (\bar N') ^n = (\tilde e_1 , \dots, \tilde e_d) = ( e_1 , \dots, e_d)\sum \limits^\infty_{n= 0} \gamma_n  (\gu) (\bar N) ^n \]
Taking monodromy on the both sides, we get $\bar N' = m \bar N$. So  $\sum \limits_{n =0} ^\infty \gamma_n ( - \gu ') (\bar N')^n = \sum \limits_{n =0} ^ \infty \gamma_n ( - m \gu ') ( \bar N)^n $.  Hence we still obtain Equation \eqref{Eq: difference of basis}:  \[ (e_1, \dots , e_d ) = (e'_1, \dots , e'_d) \sum_{n =0 } ^\infty  \gamma _n ( -\log([\unu])) \bar N ^n.\] The remaining for the proof of  the claim and the theorem is the same as that of Theorem \ref{thm: main1.5}.

\end{proof}
\subsection{Comparison between Wach modules and Kisin modules} Throughout this subsection, we assume that $K= K_0$ is unramified.  We have a natural embedding $W(k)[\![v]\!]$ to $W(R)$ via $v \mapsto [\e]-1$ and  denote $\gs_{\e}\subset W(R)$ the ring $\gs$ via the embedding  $v\mapsto [\e]-1$. Note that $\Gamma := \gal (K_{p^\infty}/K)$ acts on $ W(k)[\![v]\!]$ naturally  and  commutes with $\varphi$-action. Set $q: = \varphi (v)/v$.  Following \cite{Ber}, a \emph{Wach module of height $r$} is a finite free $\gs_{\e}$-module $\N$ with the following structure:
\begin{enumerate}
\item There exist semi-linear $\varphi$-action and  $\Gamma$-action on $\N$ such that $\varphi_\N$ and $\Gamma_\N$ commutes. 
\item The cokernel of linear map $1 \otimes \varphi_\N: \gs _{\e} \otimes_{\varphi, \gs_{\e}} \N \to \N$ is killed by $q^r$.
\item $\Gamma_\N$ acts on $\N / v\N$ trivially.
 \end{enumerate}

For any Wach module $\N$, we can attach a $\Z_p [G_K]$-module
$$T_{\t{Wa}} (\N) : = \Hom_{\gs_\e , \varphi} (\N, W(R));$$
For any $f \in T_{\t{Wa}} (\N)$, $g\in G_K$, $ g$ acts on $f$  via $(g\circ f) (x) = g(f (g^{-1} x)), \forall x \in \N$, where $G_K$ acts on $\N$ via $G_K \twoheadrightarrow \Gamma$. We note that usually one attaches  $\N$ a representation via $\tilde T (\N):= (\N \otimes_{\gs_\e} {\bf A}) ^{\varphi=1}$ (as in \S I.2 in \cite{Ber}), where  ${\bf A }$ is constructed as follows:
Let ${\E^\ur_\e}$ be the maximal unramified extension of $\E_\e$ in $W(\t{Fr} R)$, where $\t{Fr} R$ is the fraction field of $R$ and $\E_\e$ is the fraction field of  the $p$-adic completion of $W(k)[\![v]\!][\frac 1 v]$. Set ${\bf A}$ to be the $p$-adic completion of the ring of integers of $\E^\ur_\e$.     But it is well-known that $T_{\t{Wa}}$ is the dual of $\tilde T$.

Let ${\rig }$ be the ring of series $\sum\limits_{n = 0}^\infty a_n v^n , a_n \in K_0$ such that the formal series $\sum\limits _{n= 0}^\infty a_n X^n$ converges for any $x \in \hat \O_{\overline K}$ (the $p$-adic completion of $\O_{\overline K}$). Let $\tilde{\t{B}} \subset \R_{K_0}$ be the subring containing the sequence $\sum\limits_{n=0}^\infty a_n t ^{\{n\}} $. It is easy to check that $\rig \subset \tilde{\t{B}}$.

The following Theorem is a summary of properties of Wach modules that we need from \cite{Ber}:
\begin{theorem}\label{thm: WachBerger}
\begin{enumerate}
\item The functor $T_{\t{Wa}}$ induces an anti-equivalence between the category of $G_K$-stable $\Z_p$-lattices in crystalline representations with Hodge-Tate weights in $\{0, \dots, r\}$ and the category of Wach modules of height $r$.
\item Write $T : = T_\t{Wa}(\N)$. Then $T_{\t{Wa}}$ induces an injection \[\iota_{\t{Wa}} :  W(R) \otimes _{\gs_\e}\N \inj T^\v \otimes_{\Z_p} W(R)\] with the corkernel killed by $v^r$.
 \item    $D_\cris (V) = (\rig \otimes_{\gs_\e} \N) ^\Gamma$ and $ (\rig \otimes_{\gs_\e} \N)/ (\rig \otimes_{\Q_p} D_\cris (V))$ is killed by some power of
 $\prod_{n = 1}^\infty  \frac{\varphi^{n-1}(q)}{p}$.
\end{enumerate}
\end{theorem}
\begin{proof}  See Theorem 2,  Proposition II.2.1,   Proposition III.2.1,  Theorem III.3.1 in \cite{Ber}.
\end{proof}

Now we can follow the similar idea of \S\ref{subsec: 3.1} to prove Theorem \ref{Thm: Wach&Kisin}. Let $\hat e_1 , \dots, \hat e_d$ be an $\gs_\e$-basis of the Wach module $\N$ and $e_1 , \dots, e_d$  a $K_0$-basis of $D_\cris (V)$. Theorem \ref{thm: WachBerger}  (3) implies that $(e_1, \dots , e_d) = (\hat e_1 , \dots, \hat e_d) Y$ with $Y$ a matrix  having entries in $\rig$. Since $ \frac{\varphi^{n -1}(q)}{p}$ is a unit in $\tilde B$ for $n \geq 1$, $Y$ is an invertible matrix with $Y^{-1} \in \t{M}_{d \times d} (\tilde B)$.
 On the other hand, if  $\hat e'_1 , \dots, \hat e'_d $ is  an $\gs_\upi$-basis of the Kisin module $\M$, then  we have seen from \S \ref{subsec: 3.1} that $ (e_1 , \dots ,e _d ) = (\hat e'_1 , \dots , \hat e'_d) Y' $ with $Y'$ a matrix having entries in $S_\upi[\frac 1 p]$. Note that both $Y $ and $Y'$ are invertible matrices in $\t{M}_{d\times d  } (\R_{K_0})$.
  Therefore $(\hat e_1 , \dots, \hat e_d)= (\hat e'_1 , \dots, \hat e'_d) X $ with $ X= Y' {Y}^{-1}$. On the other hand,  Theorem \ref{thm: WachBerger} (2) implies that $v ^r (\hat \iota (\hat e'_1 , \dots , \hat e'_d)) \subset \iota_{\t{Wa}} (W(R) \otimes _{\gs_\e} \N)$. Therefore $v^r X $ has entries in $W(R)$. It is well-known that $v= [\e]-1$ is a generator of $I^{[1]}W(R)$. So Lemma \ref{lem: key} implies that $X$ has entries in $W(R)$. Similarly  we can show that $X^{-1}$ has entries in $W(R)$.

Now we conclude that $\hat \iota (\hR \otimes_{\varphi, \gs} \M) = \iota _{\t{Wa}} ( \hR \otimes _{\gs _\e} \N)$. To prove Theorem \ref{Thm: Wach&Kisin}, it suffices to show that $\gs_\e = (\hR)^{H_{p^\infty}}$. Since it is easy to show that  $\tilde{\t{B}} \cap W(R) = \gs_\e$, it suffices to check that
$ (\R_{K_0}) ^{H_{p^\infty}} = \tilde {\t{B}}$.

Note that the $G_K$-actions on $\R_{K_0}$ factors through $\hat G$. 
We have the following results on $G_{p ^\infty}$ and $G_\infty$-invariants of $\R_{K_0}$:

\begin{lemma}\label{lem: invariants}
$(\R_{K_0}) ^{G_{p^\infty}} = \tilde{\t {B}} $ and $(\R_{K_0}) ^{G_\infty} = S[\frac 1 p]$.
\end{lemma}
\begin{proof}
 We first show that $(\R_{K_0}) ^{G_{p^\infty}} = \tilde{\t {B}} $. First assume that $p> 2$.   Since $\hat G \simeq G_{p ^\infty} \rtimes H_K$ by Lemma 5.1.2 in \cite{liu3}, we can pick a $\tau \in G_{p ^\infty}$ such that $\tau$ is a topological generator of $G_{p ^\infty}$ and $[\e(\tau )] = \exp(-t)$. For any $x \in \R_{K_0}$, by the definition of $\R_{K_0}$, we may write $x = \sum_{i = 0} ^\infty f_i u ^i$ with $f_i \in \tilde{\t{B}}$. It suffices to show that $f_i = 0$ for any $i > 0$. Note that $\tau$ acts on $\tilde{\t{B}}$ trivially and $\tau (u) = u [\e(\tau)]= u \exp (-t)$. Hence $\tau ( x) = \sum_{i=0} ^\infty f_i (\exp (-t))^i  u ^i$. So by Lemma 7.1.2 in \cite{liu2},  $x \in (\R_{K_0}) ^ {G_{p ^\infty}}$ implies that $ f_i (\exp(-t)) ^ i = f_i$ for all $i$. Therefore $f_i = 0$ unless $i = 0$. If $p =2$, then \S 4.1 in \cite{liu4} shows that we can pick a $\tau \in G_{p ^\infty}$ such that $ [\e (\tau)]  = \exp(-2t)$. The remaining proof can proceed the same as before.

 For the proof of $(\R_{K_0}) ^{G_\infty} = S[\frac 1 p]$, we use the essentially the same idea. For any $x \in \R_{K_0}$, we can write $x = \sum_{j = 0} ^\infty f_j t ^j$ with $f _j \in S[\frac 1 p]$. For any $g \in G_\infty$, $g(u) = u$ and $g (t) = \chi_p (g) t$ where $\chi_p $ is the $p$-adic cyclotomic character. Then the statement that  $(\R_{K_0}) ^{G_\infty} = S[\frac 1 p]$  again follows  Lemma 7.1.2 in \cite{liu2}.

\end{proof}
\section{Applications to de Rham representations}

\subsection{Various lattices in $D_\st (V)$ }\label{susec: 4.1}Let $T$ be a $G_K$-stable $\Z_p$-lattice inside a semi-stable representation $V$ of $G_K$ with Hodge-Tate weights in $\{0 , \dots , r\} $. By using Kisin modules or its variation, we can attach the following  $\varphi$-stable $W(k)$-lattices  (related to $T$) in $D_\st (V)$ as the following:  Let $\hM= (\M , \varphi, \hat G)$ be the  $(\varphi , \hat G)$-module attached to $T$,  $\D = S[\frac 1 p ] \otimes_{\varphi, \gs}\M$ and $D := \D/ I_+S \D$. Recall there exists a unique  $(\varphi, N) $-equivariant section $s: D \to \D$. By Proposition 2.2.2 in \cite{liu-wd}, there exists a unique isomorphism of $W(k)$-modules $i :  D_\st (V) \simeq s(D)$ to make  Diagram \eqref{i} commutes. Now we can define $$M_\st (T) : = ( i^{-1} \circ s)( \M / u \M) \subset D_\st (V)$$ as  in \cite{liu-wd}, \S2.3. On the other hand, set $ \CM= \CM_S (\M) = S \otimes_{\varphi, \gs} \M \subset \D$, we can define $$\tilde M_\st(T):=  i ^{-1}(s(D) \cap \CM).$$

Let $\varpi \in \O_{K}$ and $\underline \varpi = (\varpi_n) \in R$ with $\varpi_n$ a $p^n$-th root of $\varpi$. Set $\mathfrak v: = \log ([\underline \varpi])$ and $A^+_\st:= \acris [\gv]$. It is obvious that $A^+_\st[\frac 1 p] = B^+_\st$ and the construction depends on the choice of $\gv$. If we define the monodromy operator $N $ on $B^+_\st$ via $N(\gv) = 1$ then we see that $A^+_\st$ is $G_K$-stable, $\varphi$-stable and $N$-stable inside $B^+_\st$. Define \[M_\inv (T) := ( T ^\v \otimes_{\Z_p} A^+_\st) ^{G_K}.\] If $V$ is crystalline then $M_\inv (T) = (T^\v \otimes_{\Z_p} \acris)^{G_K}$ and the construction of $M_\inv(T)$ does not depend on the choice of $\varpi$ in this case.
\begin{remark} \label{remark: change N} According to Remark \ref{remark set N},  the integral theory via Kisin modules or Breuil modules uses the convention $N(\gu) =1$.   So if we set $N(\gv) =1$ as the above then we change the monodromy setting of Breuil-Kisin theory. But luckily, the construction of $M_\inv$ does not depend on Breuil and Kisin's theory.

\end{remark}

The following Proposition summarizes some properties of these lattices. 

\begin{prop}\label{prop: intersection} \begin{enumerate}
\item $\tilde M_\st(T) \subset M_\st(T)$. There exists a constant $c_1$ depending on  $e= [K:K_0]$ and $r$ such that $p^{c_1}M_\st (T) \subset \tilde M_\st(T)$.

\item  Assume that $V$ is crystalline.  Then $\tilde M_\st (T) \subset M_\inv (T)$.  There exists a constant $c_2$ depending on $e$ and $r$ such that $p^{c_2} M_\inv (T) \subset \tilde M_\st(T)$.
\item $M_\st (T)$ is  $N$-stable inside $D$.  So is $\tilde M_\st (T)$ if $p > 2$.
\item $M_\inv (T)$ is $\varphi$-stable  and $N$-stable inside $D_\st(V)$.
\item  The functor $M_\inv: T \mapsto M_\inv (T)$ is left exact.
\item If $e=1$, $ r \leq p-2$ and $V$ is crystalline then $M_\st = \tilde M_\st = M_\inv$.

\end{enumerate}
\end{prop}


\begin{proof} (1)  Write $\tilde M = \tilde M_\st (T)$,  $q : \D \twoheadrightarrow D$ and $M = s \circ q(\CM)$.  It is easy to check that $\M / u\M = q (\CM)$ inside $D$. So it suffices to show that $\tilde M \subset M$.  Note that $s\circ q(x) = x $ for any $x \in s(D). $ Since $\tilde M \subset \CM $, we see that $ \tilde M = s\circ q (\tilde M) \subset s\circ q (\CM) = M.  $ To show the existence of the constant $c_1$ it suffices to show that there exists a constant $c_1$ such that $p ^{c_1} M \subset \CM$ and this has been proved in Lemma 7.3.1. in \cite{liu2}.

(2) We regard $ \D $ as a submodule of $V^\v \otimes_{\Z_p} \acris$ via the  injection $\iota: \acris \otimes_S \D \inj V^\v \otimes_{\Z_p} \acris$ by Theorem \ref{thm: big-review} (2). It is easy to check that $\CM \subset T^\v \otimes_{\Z_p} \acris$ by Theorem \ref{thm: big-review} (3). By the construction of isomorphism $i$ in Diagram \eqref{i}, $s(D) = D_\cris (V)$ if $V$ is crystalline. So we have that $\tilde M_\st (T) \subset M_\inv (T) $. Let $e_1, \dots, e_d $ be a $W(k)$-basis of $\tilde M = \tilde M_\st(T)$. For any $x \in M_\inv (T)$ we may write $x = \sum_i a_i e_i $ with $a_i \in K_0$. By Lemma 5.3.4. in \cite{liu2}, $t^r x \in \acris \otimes_S \CM$. By the proof of Proposition 2.4.1. in \cite{liu6}, we see that there exists a constant $c_3$ depending on $e$ and $r$ such that $p^{c_3}  \acris \otimes_S \CM \subset \acris \otimes _{W(k)} M$. By (1), we may assume that $p^{c_3}  \acris \otimes_S \CM \subset \acris \otimes _{W(k)} \tilde M$. That is, $p^{c_3} t^r x= \sum _i a_i  p^{c_3} t^r e_i \in \acris \otimes _{W(k)} \tilde M$. So $p^{c_3} t ^r a_i \in \acris$. Let $c_4 $ be the largest integer depending on $r$  such that $t^r / p^{c_4} \in \acris$. Then we see that $c_2 = c_3 + c_4$ is the required constant.

(3) These are consequences of Proposition 2.4.3,  and Proposition 2.4.1 in \cite{liu-wd}. Note that  Proposition 2.4.1 requires $p >2$.

(4) and (5) are obvious from the construction. We note that (5) is different from the statement (3) because we change the $N$-structure on $B^+_\st$ by setting  $N(\gv) = 1$.

(6) In this situation, the $G_K$-stable $\Z_p$-lattices can be studied by  Fontaine-Laffaille's theory in \cite{FL82}.  Recall a \emph{strongly divisible $W(k)$-lattice} $(L, \Fil ^i L , \varphi_i)$ is a finite free $W(k)$-module $L$ with the following structures:
\begin{itemize}
 \item A filtration $\Fil ^i L \subset L$ such that $\Fil ^0 L= L$,  $\Fil ^{p-1} L = \{0\}$ and $L/ \Fil ^iL$ is torsion free.
 \item $\varphi_i : \Fil ^i L \to L$ is a Frobenius semi-linear map such that $\varphi _i |_{\Fil ^{i+1 }L} = p \varphi_{{i+1}}$.

 \item  $\sum_{i = 0}^{p-2} \varphi_i (\Fil ^i L) = L.$
\end{itemize}
Since $p ^i | \varphi (\Fil ^i \acris)$ in $\acris$ for $ 0 \leq i \leq p-1$, one can define $\varphi_i := \varphi/ p^i : \Fil ^{i}\acris \to \acris$. By the main result in \cite{FL82}, there exists a strongly divisible $W(k)$-lattice $(L, \Fil ^i L, \varphi_i)$ such that $\Hom_{W(k), \Fil ^i , \varphi_i}  (L, \acris) \simeq T$ and $L [\frac 1 p ]= D_\cris (V)$ as filtered $(\varphi, N)$-modules (we can define $\varphi$ on $L$ by $\varphi = \varphi _0$). 
On the other hand, define $\CL := S \otimes_{W(k)} L$, $\Fil ^{r}\CL: = \sum_{i= 0}^r  \Fil ^i S \otimes \Fil ^{r-i}L$, and a semi-linear map $\varphi_r: = \sum_{i= 0}^r  \varphi_{i, S} \otimes \varphi_{r-i , L}: \Fil ^r \CL \to \CL$, where $\varphi _{i , S} := \varphi_S / p^i : \Fil ^i S \to S$ and $\varphi_{r-i , L}= \varphi_{r-i}: \Fil ^{r-i} L \to L$.
It is easy to check that $(\CL , \Fil ^r \CL , \varphi_r)$ is a \emph{quasi-strongly divisible $S$-lattice}\footnote{Here we use ``$S$-lattices" to distinguish strongly divisible $W(k)$-lattices in Fontaine-Laffaille's theory.} inside $\D$ in the sense in \cite{liu3} (see Definition 2.3.3).  On the other hand, $\CM= \CM_\gs (\M)$ is also a quasi-strongly divisible $S$-lattice inside $\D$, which is the key point in \cite{liu3}, \S 3.4. For any quasi-strongly divisible $S$-lattice $\CN$ inside $\D$, Proposition 3.4.6 {\it{loc. cit.}} show that the functor
$$T_\cris : \CN \mapsto \Hom_{S, \Fil ^r, \varphi_r} (\CN, \acris)$$
establishes an anti-equivalence between the category of quasi-strongly divisible $S$-lattices and the category of $G_\infty$-stable $\Z_p$-lattices inside semi-stable representations with Hodge-Tate weights in $\{0, \dots, r\}$. Now we claim that $T_\cris (\CL) \simeq T_\cris (\CM)$ as $\Z_p [G_\infty]$-modules and consequently $\CL \simeq \CM$. Indeed, it is straightforward to check that $T_\cris (\CL) \simeq \Hom_{W(k), \Fil ^i , \varphi_i} (L , \acris) |_{{G_\infty}}\simeq T|_{G_\infty}$. On the other hand, combing Lemma 3.3.4 in \cite{liu2} and Theorem \ref{review} (3), we see that $T_\cris (\CM) \simeq T_\gs (\M) \simeq \hat T (\hM)|_{G_\infty}= T|_{G_\infty}$. So $T_\cris (\CL) \simeq T_\cris (\CM)$. In summary, there is an $S$-linear isomorphism $\CM\simeq S \otimes_{W(k)} L$ compatible with $\varphi$-structures. Recall that $s : D \to \D$ is unique $\varphi$-equivariant section for the projection $\D \twoheadrightarrow D$. So we conclude  that $s\circ q (\CM) = L = M_\st (T) = \tilde M_\st (T)$.

It remains to show that $M_\inv (T) = M_\st(T)$. The idea is the same as the proof of existence of $c_2$ ($c_2 = 0$ in this case).
Let $e_1, \dots, e_d $ be a $W(k)$-basis of $L  = \tilde M_\st(T)$. For any $x \in M_\inv (T)$ we may write $x = \sum_i a_i e_i $ with $a_i \in K_0$. Lemma 5.3.4. in \cite{liu2} showed that  $t^r x \in \acris \otimes_S \CM= \acris \otimes_{W(k)}L $.  Hence $ t^r x= \sum _i a_i   t^r e_i \in \acris \otimes _{W(k)} L$. So $ t ^r a_i \in \acris$. As $r \leq p-2$, $a_i$ has to be in $W(k)$ to make $a_i t^r \in \acris$.
\end{proof}

\begin{remark} The functor $M_\st$ enjoys some nice properties. For example, it is useful to study torsion representations discussed in \cite{liu-wd} and \cite{liu6}, and it is  compatible with tensor products. But $M_\st$ does not have good exact properties where $M_\inv$ is  left exact. And this is important for \S\ref{subsec: last}.
\end{remark}

\begin{example}\label{ex: 2} Unfortunately, the functor $M_\st$ is not left exact as claimed in Theorem 2.1.3 (the remaining of the theorem is still correct). Indeed,  Example 2.5.6. in \cite{liu-wd} just serves the example that $M_\st$ neither left exact nor right exact. For convenience of the readers, we repeat the example here. Let $K = \Q_p( \pi)$ with $\pi^{p-1} = p.$ Set  $E(u) = u^{p-1}-p.$ Let $\M$ be the  rank-2 Kisin module given by $\varphi (e_1)= e_1$ and $\varphi (e_2)= u e_1 + E (u)e_2$ with $\{e_i\} $ an $\gs $-basis of $\M$. Let $\gs ^*= \gs\cdot \fe$ be the rank-1 Kisin module with $\fe$ the basis and $\varphi (\fe) = E(u)\fe.$ Consider the sequence of Kisin modules
\begin{equation}\label{Eq: example}
0 \to \gs ^* \overset {\mathfrak i} {\to} \M \overset \gf \to \gs \to 0
\end{equation}
 where $\gf$ and $\mathfrak i $ is induced by  $\gf(e_1) = p$ and $\gf(e_2) = u$ and $\mathfrak i (\fe)= u e_1 -p e_2$. It is easy check the sequence is a left exact sequence of Kisin modules with height 1 and the sequence is exact after tensoring $\O_\E$, which is the $p$-adic completion of $\gs [\frac 1 u]$.  As explained above Example 2.5.5 in \cite{liu-wd} and Lemma 2.5.4. {\it loc. cit.}, Theorem (0.4) in \cite{kisin2} implies that the above sequence of Kisin modules can be naturally be extended to sequence of $(\varphi, \hat G )$-modules, and $\hat T$ of the sequence is an \emph{exact} sequence of $G_K$-stable $\Z_p$-lattices in crystalline representations with Hodge-Tate weights in $\{0, 1\}$:
$$ 0 \to \Z_p \to T \to \Z_p (1) \to 0.$$
Now modulo $u$ to the sequence in \eqref{Eq: example}, we get the sequence of $W(k)$-modules:
$$0 \to W(k) \cdot \bar \fe\overset i  \to M  \overset f \to W(k) \to 0 , $$
where  $M = \M/ u \M\simeq M_\st (T)$. We can easily check that the above sequence is not exact on $M$ and $W(k)$. Hence the functor $M_\st$ is not left exact according to the construction of $M_\st$.

\end{example}
\subsection{Various lattices in $D_\dR (V)$} Let $V$ be a De Rham representation of $G_K$ with Hodge-Tate weights in $\{0, \dots, r \}$ and $T$ a $G_K$-stable $\Z_p$-lattice inside $V$. It has been proved that $V$ is potentially semi-stable (\cite{Ber0}). Let us assume that $V$ is semi-stable over $K'$, which is a finite and Galois over $K$. Let $k '$ be the residue field of $K'$ and $K'_0 := W(k')[\frac 1 p]$.


Set $D_\dR (V):  = (V^\v \otimes_{\Q_p} B^+_\dR)^{G_K} $ and $D_{\st, K'}:= ( V^\v  \otimes_{\Z_p}  B_\st ^+)^{G_{K'}}$.    It is well-known that  $ D_{\dR, K' }(V):= (V^\v \otimes_{\Q_p} B^+_\dR )^{G_{K'}} =  K' \otimes _{K'_0} D_{\st, K '} (V) $ and $D_{\st, K'}$ has a semi-linear $\gal(K'/K)$-action.  Let $M_{*, K'} (T) \subset D_{\st , K'} (V)$ denote the lattices $M_\st $, $\tilde M_\st$ and $M_\inv $ constructed in \S \ref{subsec: 3.1} for $T |_{G_{K'}}$. 

We define one more lattice before discussing the properties of $M_{*, K'}$. Let $ ( \M, \varphi, \hat G)$, $\CM \subset \D$ and $D$ denote the data attached to $T|_{G_{K'}}$ as in the beginning of \S \ref{subsec: 3.1}.  Since $\D = S  \otimes_{W(k)} s(D)$ via section $s$ and $D$ is isomorphic to $D_{\st, K'}(V)$ via the isomorphism $i $ in Diagram \eqref{i},  we may identify $\D/ \Fil ^1 S \D $ with $D_{\dR, K '}(V)$. Set $M_{\dR, K'} (T) = \CM / \Fil ^1 S \CM \subset \D / \Fil ^1 S \D \simeq D_{\dR, K '}(V)$ and $M_\dR(T):= D_\dr (V) \cap M_{\dr, K'}(T)$.

The following  proposition  shows that the constructions of $M_{\st, K'}$ and $M_{\dr , K'}$ do not depend on the choice of uniformizer $\pi \in \O_{K'}$. 

\begin{prop}\label{prop: intersection2} Notations as the above, constructions of $M_{\st, K'}$ and $M_{\dr , K'}$ do not depend on the choice of uniformizer $\pi \in \O_{K'}$. If $V$ is potentially crystalline then $\tilde M_{\st , K'}$ and $M_{\inv, K'}$ also do not depend on the choice of uniformizer $\pi \in \O_{K'}$.
\end{prop}

\begin{proof} Since we only use $G_{K'}$-structure in the following proof, without loss of generality, we may assume that $K= K'$. Suppose that we select another uniformizer $\pi' \in \O_{K}$ and the embedding $\gs \subset W(R)$ via $u \mapsto [\upi']$. We add $'$ to all the data for the choice uniformizer $\pi'$ and the embedding $\gs \subset W(R)$ via $u \mapsto [\upi']$. We note that the embedding
$s: D \subset  \D \subset T^\v \otimes_{\Z_p} B^+_\cris$ indeed depends on such choice because the isomorphism $i ^{-1} : s(D) \simeq D_\st (V)$ is given by $y \mapsto \sum \limits^\infty_{n=0} N^i (y) \otimes \gamma_ i (\gu)$, unless $N=0$ or, equivalently, $V$ is crystalline. So we label the isomorphisms $i_{\upi} : D_\st (V) \simeq s(D)$ and $i _{\upi'}: D_\st (V) \simeq s'(D')$ to distinguish them. Recall that $I_+R \subset R$ is the maximal ideal of $R$. Let $\nu: W(R) \to W(\bar k)$ be the projection induced by modulo $W(I_+R)$. \S3.2 in  \cite{liu-wd} showed that $\nu$ can be naturally extended to $\nu: B^+_\st \to \tilde K_0$ such that  $\nu(\gu)= 0$, where $\tilde K_0 = W(\bar k)[\frac 1 p]$. We write $I_+ := \t{Ker} (\nu)$.

From the construction of $M_\st$, we have the following commutative diagram:
$$\xymatrix@C=30pt{ D_\st (V) \ar[d]^{\wr}_i  \ar@{^{(}->}[rr]& & B^+_\st \otimes_S \D\ar@{->>}[d]^{\mod \gu}  \ar@{->>}[r]^-{\mod I_+} &  \tilde K_0 \otimes_{K_0} D_\st (V) \ar@{=}[d]  & \\
s(D) \ar@{^{(}->}[r]^ s  & \D \ar@{^{(}->}[r] & B^+_\cris \otimes_S \D \ar@{->>}[r]^-{\mod I_+} & \tilde K_0 \otimes_{K_0} D \\
 & \M \ar@{^{(}->}[u]  \ar@{^{(}->}[r]& W(R) \otimes_{\varphi, \gs} \M \ar[ur]^\gamma \ar@{^{(}->}[u] \ar@{->>}[r]^-{\mod I_+} & W(\bar k) \otimes _{W(k)} \M/u \M \ar@{^{(}->}[u]}
$$
Here $\gamma$ is just the composite of maps $\xymatrix{W(R)\otimes_{\varphi, \gs} \M  \ar@{->>}[r]^-{\mod I_+} & W(\bar k) \otimes _{W(k)} \M/u\M } $ and $W(\bar k) \otimes _{W(k)} \M/u \M  \inj \tilde K_0 \otimes_{K_0}D$.
Let us write $\alpha$ for  the composite of maps in the first row  of the above diagram. It is obvious that the first row (hence $\alpha$) is independent on the choice of uniformizer $\pi$, while the second and third rows are. The above diagram and the construction of $M_\st (T) = i^{-1} \circ s (\M / u \M)$ shows that $M_\st (T) = (\alpha ^{-1}\circ  \gamma) (W(R) \otimes_{\varphi, \gs} \M) $.

Now we select another uniformizer $\pi' \in \O_{K'}$ and the embedding $\gs \subset W(R)$ via $u \mapsto [\upi']$. We still get the above diagram and $M'_\st (T) = (\alpha ^{-1} \circ \gamma')  (W(R) \otimes_{\varphi , \gs_{\upi'}} \M')$. By Theorem \ref{main1}, we have $W(R) \otimes_{\varphi , \gs_{\upi'}} \M' = W(R) \otimes_{\varphi , \gs_\upi} \M$ as submodules of $T^\v \otimes_{\Z_p} W(R)$. Hence $\gamma ( W(R) \otimes_{\varphi , \gs_\upi} \M ) = \gamma' ( W(R) \otimes_{\varphi , \gs_{\upi'}} \M')$. Since $\alpha$ is independent on the choice of $\pi$, we conclude that
$M_\st (T) = M'_\st(T)$.

 We use the similar idea as the above to show that $M_\dR (T)$ does not depend on the choice of $\pi$. For any subring $B \subset B^+_\dR$, recall that $\Fil ^1 B = \Fil ^1 B_\dR^+ \cap B$. For any ring $A \subset B^+_\st $ such that $W(k) \subset A$, we have a natural map $$\theta:  A \otimes _{W(k)} \O_K \subset B^+_\st \otimes_{K_0} K \subset B^+_\dR \to B^+_\dR/ \Fil ^1 B^+_\dR = \C_p $$ induced by modulo $\Fil ^1$.
  Now according to  the construction of $M_\dR (T)$, we can modify the above diagram as the following:

 $$\xymatrix@C=20pt{ D_\dR  (V) \ar[d]^{\wr}_i  \ar@{^{(}->}[rr]& & K \otimes_{K_0}B^+_\st \otimes_S \D\ar@{->>}[d]^{\mod \gu}  \ar@{->>}[r]^-{\mod \Fil ^1 } &  \C_p \otimes_{K } D_\dR (V) \ar@{=}[d]  & \\
K \otimes_{K_0}s(D) \ar@{^{(}->}[r]^ -{s_K}  & K \otimes_{K_0}\D \ar@{^{(}->}[r] & B^+_\cris \otimes_S (K \otimes_{K_0}\D)  \ar@{->>}[r]^-{\mod \Fil ^1} & \C_p \otimes_{K}(K \otimes _{K_0} D) \\
 & \CM \ar@{^{(}->}[u]  \ar@{^{(}->}[r]& A_\cris \otimes_{S} \CM \ar[ur]^{\gamma_K} \ar@{^{(}->}[u] \ar@{->>}[r]^-{\mod \Fil ^1} &  \O_{\C_p} \otimes _{\O_K} \CM/\Fil ^1 \CM \ar@{^{(}->}[u]}
$$
We see that the map $\alpha_K$ in the first row is still independent on the choice of uniformizer $\pi$ and $M_\dR(T) := (\alpha_K^{-1} \circ \gamma_K) (\acris \otimes_S \CM) $. Now repeat the proof of that $M_\st$ does not depend on the choice of $\pi$, we conclude that $M_\dR (T)$ does not depend on the choice of $\pi$.

When $V$ is crystalline (as we assume that $K' = K$), we see that $s(D) = D_\cris (V)$ which does not depend on choice of $\pi$. It is clear from the construction of $\tilde M_\st$ that $\tilde M_\st (T) = s(D) \cap (\acris \otimes_S \CM)$. Since  $\acris \otimes_{\varphi , \gs_{\upi'}} \M' = \acris \otimes_{\varphi , \gs_\upi} \M$ as submodules of $T^\v \otimes_{\Z_p} W(R)$ by Theorem \ref{main1}, we conclude that $\tilde M_\st$ does not depend on the choice of uniformizer $\pi$. Finally, it is obvious from the construction that $M_\inv (T)$ does not depend on the choice of $\pi$ or $\varpi$ if $V$ is crystalline.

\end{proof}

We wish to discuss the formation of those functors when base changes. Let $K'' /K '$  be a finite extension,  $k''$ the residue field of $K''$ and $\O_{K''}$  the ring of integers.

\begin{prop}\label{prop: basechange} \begin{enumerate}\item $M_{\st , K''} (T) = W(k'') \otimes_{W(k')} M_{\st, K'} (T). $
\item $M_{\dR , K''} (T) = \O_{K''} \otimes_{\O_{K'}} M_{\dR, K'} (T)$.
\end{enumerate}
\end{prop}
\begin{proof} To prove (1) and (2), we use almost the same ideas as the proof of the above proposition. By the first commutative diagram in the above proof,  $M_{\st, K''}  (T)  = (\alpha^{-1}\circ \gamma ) (W(R) \otimes_{\varphi , \gs} \M'') \subset D_{\st , K''}(V)$. By Theorem \ref{thm: main1.6},
$W(R) \otimes_{\varphi , \gs_{\upi''}} \M'' = W(R) \otimes_{\varphi , \gs_{\upi'}} \M'$  as submodules of $T^\v \otimes_{\Z_p} W(R)$. Hence we have   $M_{\st , K '} (T) $ is just $(\alpha^{-1}\circ \gamma ) (W(R) \otimes_{\varphi , \gs} \M'')$ restricted to $D_{\st, K'}(V)$. That is, $M_{\st , K'} (T) = M_{\st, K''} (T) \cap {D_{\st, K'}(V)}$. Since $M_{\st, K' }(T)$ is a $W(k')$-lattice inside  $D_{\st, K'} (V)$ and $D_{\st , K''} (V) = W(k'') \otimes_{W(k')} D_{\st, K'} (V)$, we get $ W(k'') \otimes_{W(k')} M_{\st, K'} (T)\subset M_{\st , K''} (T)$. But
\begin{eqnarray*}
W(\bar k) \otimes_{W(k')}\M'/ [\upi'] \M'   =  W(R) \otimes_{\varphi, \gs_{\pi'}} \M' \mod W(I_+ R) \\ =   W(R) \otimes_{\varphi, \gs_{\pi''}} \M'' \mod W(I_+R) =W(\bar k) \otimes_{W(k'')}\M''/ [\upi''] \M''.
\end{eqnarray*}
Hence $M_{\st, K'}(T)$ must generate $M_{\st, K''}(T)$ as $W(k'')$-modules and then we conclude $ W(k'') \otimes_{W(k')} M_{\st, K'} (T)= M_{\st , K''} (T)$. The proof of (2) proceeds totally similarly.
\end{proof}



If $V$ is semi-stable non-crystalline then $M_{\inv, K}$ in general does depend on the choice of $\varpi$. To study de Rham representations  of $G_K$ in the next subsection,   we fix a uniformizer $\varpi$ of $\O_K$ to define $A_\st$ from now. Set
$$\tilde M_\inv (T) := D_\dr(V) \cap (\O_{K'} \otimes_{ W(k')} M_{\inv, K'}(T)) = ( \O_{K'} \otimes_{W(k')} M_{\inv, K'}(T))^{\gal (K'/K)}. $$
It is easy to see that $\tilde M_\inv (T) $ is an $\O_K$-lattice in $D_\dr(V)$ but lost the $(\varphi, N)$-action. The following lemma summarizes the useful properties of $\tilde M_\inv$.
\begin{lemma} \label{lem: property of inv}\begin{enumerate}
\item Let $K''/K'$ be a finite extension  with the residue field $k''$. Then $M_{\inv , K''} (T) = W(k'') \otimes_{W(k')} M_{\inv, K'} (T)$.
\item The construction of $\tilde M_\inv$ does not depend on the choice of $K'$.
\item The functor $\tilde M_\inv $ is left exact.
\item Assume that $T_1^\v  \inj T_2 ^\v$ is an injection of two $G_K$-stable $\Z_p$-lattices of De Rham representations. If $p ^a$ kills the torsion part of $T^\v_2/ T^\v_1$ then $p^a$ kills the torsion part of $\tilde M_\inv (T_2)/ \tilde M_\inv (T_1)$.

\end{enumerate}
\end{lemma}

\begin{proof} (1) It is clear that  $M_{\inv, K ''} (T)[\frac 1 p] \simeq W(k'') \otimes_{W(k')} M_{\inv, K'}(T) [\frac 1 p]$. So we have
$ W(k'') \otimes_{W(k')} M_{\inv, K'}(T) \subset M_{\inv, K'' }(T)$  as lattices inside $M _{\inv, K''}(T)[\frac 1 p]$. Hence to prove (1) it suffices to assume that $K''/ K'$ is Galois. Write $D'':= M_{\inv, K''} (T) [\frac 1 p]$, $D' = M_{\inv, K'} (T) [\frac 1 p]$, $K''_0 = W(k'')[\frac 1 p]$ and $K'_0 = W(k')[\frac 1 p]$. We see that $D''$ has a semi-linear $\gal(K''/K')$-action and $D'' \simeq K''_0 \otimes _{K'_0}D'$ as $\gal (K''/ K')$-modules, where $\gal (K''/K')$  acts on $D'$ trivially.  It is obvious that the $\gal(K''/K')$-action factors through $\Gamma:= \gal (K''_0 /K'_0)$, $M_{\inv, K''}(T)\subset D''$ is $\Gamma$-stable and $M_{\inv, K'}(T) = M_{\inv, K''}(T)^{\Gamma} $. Then $M_{\inv, K''}(T) = W(k'') \otimes _{W(k')} M_{\inv, K'}(T)$ by the \'etale descent.

(2) Suppose that $K''$ is another Galois extension of $K$ such that  $V$ is semi-stable over $K''$. Let $k''$ be the residue field of $K''$.   We need to show that
\begin{equation}\label{Eq: invariants}
( \O_{K'} \otimes_{W(k')} M_{\inv, K'}(T))^{\gal (K'/K)} = ( \O_{K''} \otimes_{W(k'')} M_{\inv, K''}(T))^{\gal (K''/K)}.
\end{equation}
Without loss of generality, we can assume that $K ' \subset K''$. As $\gal (K'' / K)$-module, (1) shows that
\[ \O_{K''} \otimes_{W(k'')} M_{\inv, K''}(T) \simeq  \O_{K''} \otimes_{\O_{K'}}( \O_{K'} \otimes_{W(k')} M_{\inv, K'}(T) ).\]
Note that  $\gal (K''/ K ')$ acts on $( \O_{K'} \otimes_{W(k')} M_{\inv, K'}(T) )$ trivially, we obtain
\begin{eqnarray*}
(\O_{K''} \otimes_{W(k'')} M_{\inv, K''}(T))^{\gal (K''/K' )} &= & (\O_{K''}) ^{\gal(K''/K ')} \otimes_{W(k')} M_{\inv, K'}(T)  \\
&= & \O_{K'} \otimes_{W(k')} M_{\inv, K'}(T).
\end{eqnarray*}
Then Equation \eqref{Eq: invariants} follows by taking $\gal(K'/K)$-invariants by the both sides of the above equation.

(3) Suppose that we are given an  exact sequence $0 \to T'  \to T \to T'' \to 0$.  Applying functor $M_{\inv, K'}$, we obtain a left exact sequence
$0 \to M ''  \to M  \overset f  \to M' $ by Proposition \ref{prop: intersection} (5). We can decompose the above  sequence by two sequences
$ 0 \to M'' \to M \overset g \to N \to 0$ and $ N \overset i \inj M'$ such that the first sequence is exact. We note that the $N$ is a  finite free $W(k')$-modules as $N$ is a submodule of $M'$. In the following, we denote $A_{K'} := \O_{K'} \otimes_{W(k')} A$ for a $W(k')$-module $A$.
Since $\O_{K'}$ is flat over $W(k')$,  we still get the exact sequence
$ 0 \to M''_{K'} \to M_{K'} \overset {g_{K'}}{\to} N_{K'} \to 0$  and  $N_{K'} \overset{i_{K'}}{\inj} M'_{K'}$. Taking $\gal (K' /K)$-invariants, we obtained   a left exact sequence $ 0 \to \tilde M_{\inv} (T'') \to \tilde M_{\inv} (T) \overset{ g_{K'}}{ \to}  (N_{K'})^{\gal (K' /K)} $ and $(N_{K'})^{\gal (K' /K)}  \inj \tilde M_\inv(T'). $  Write $f_{K'}: \tilde M_\inv (T) \to \tilde M_\inv (T')$. We can easily check that  $\t{Ker} (f_{K'}) = \t{Ker} (g_{K'})$. Hence the sequence  $0 \to \tilde M_{\inv} (T'') \to \tilde M_{\inv} (T)  \to \tilde M_\inv (T')$ is left exact.

 (4) Since $\tilde M_\inv $ is left exact, without loss of generality, we can assume that $T_2^\v/ T_1 ^\v  $  is killed by $p^a$. It is obvious from the construction of $M_{\inv , K'}$ to see that $ M_{\inv, K' } (T_2)/  M_{\inv, K'} (T_1) $ is killed by $p^a$. After tensoring $\O_{K'}$ and taking Galois invariants, it is trivial to check that $p^a $ kills  $ \tilde M_{\inv} (T_2)/ \tilde M_{\inv} (T_1) $.

\end{proof}

It is easy to see that $\tilde M_\inv (T) \subset (T^\v \otimes_{\Z_p} A_\st \otimes_{W(\bar k)} \O_{\overline K})^{G_{K}}$. But we do not know how to prove that $\tilde M_\inv (T) = (T^\v \otimes_{\Z_p} A_\st \otimes_{W(\bar k)} \O_{\overline K})^{G_{K}}$.

\subsection{The direct limit of de Rham representations}\label{subsec: last}
Let $I$ be a partial order set and  $\{L_i\}_{i \in I}$ be a family of $G_K$-stable  $\Z_p$-lattices  in de Rham representations $V_i$ of $G_K$ with Hodge-Tate weights in $\{-r , \dots , 0 \}$ ($r$ is independent on $i$). Let   $L := \varinjlim_{i \in I} L_i$ be the direct limit.  Fix a uniformizer $\varpi \in \O_K$ as the last subsection to define $A_\st$ and $\tilde M_\inv$. We define a covariant version of $\tilde M_\inv $ via $\tilde M^*_\inv (T):  =  \tilde M_\inv (T ^\v)$.   Set $M_i = \tilde M^*_\inv  (L_i )$ and $ M: =\varinjlim _{i \in I} M_i$.

Recall that $L$ is $p$-adically  separated if $L$ injects to the $p$-adic limits $\hat L  $ of $L$, or equivalently, $\bigcap\limits_{n =1}^\infty p^n L = \{0\}$.

\begin{prop}\label{prop: directlimit} If $L$is $p$-adically separated then $M$ is $p$-adically separated.
\end{prop}
\begin{proof} Write  $f_{ij}: L_i \to L_j $ and $g_{ij} = \tilde M^*_\inv (f_{ij}) : M_i \to M_j$. Note that $ L_j / f_{ij} (L_i)$ and $M_j / g_{ij} (M_i)$ may not be torsion free.
Pick a $y \in M_i$ such that $ g_i (y) \not = 0$ where $g_i : M_i \to M$ is the natural map. We need to show that $ g_i (y) \not \in \bigcap\limits_{n =1}^\infty p^n M$. Suppose that  $g_i (y) $ is in $\bigcap\limits_{n =1}^\infty p^n M$. Then there exists a  subset $J=\{j_n\} \subset I$ with   $j _n < j_{n+1}$ and  $ y_n \in M_{j_n}$ such that  $p^ny_n =  g_{ij_n} (y)$. Consider the space $\t{Ker}(f_{ij_n})  \subset L_{i} $, which is an  increasing sequence of  saturated finite free $\O_K$-modules inside $L_i$. So they have to be stable after deleting finite many $j_n$. Hence without loss of generality, we may assume that all $\t{Ker}(f_{ij_n})$ are the same and then  $f_{ij_n} (L_i)$ are all isomorphic. Now we  decompose $f_{ij_n }: L _i \to L_{j_n}$ to $ L_i \overset {\tilde f_n} {\to} f_{ij_n } (L_i) \overset {\alpha_{n}}{\hookrightarrow} L_{j_n}$ and apply the functor $\tilde M^*_\inv $.  Then we get the map $ M_i \overset{\tilde g_{n}}{\to} \tilde M^*_\inv(f_{ij_n} (L_i)) \overset {\tilde \alpha_{n}}{\hookrightarrow} M_{j_n} $ where $\tilde g_{n}= \tilde M^*_\inv (\tilde f_{n})$ and $\tilde \alpha_{n} = \tilde M^*_{\inv}(\alpha_{n})$. Note that $\tilde \alpha_{n}$ is an injection because $\tilde M_\inv$ is left exact. Since  $g_{ij_n} = \tilde \alpha_n \circ \tilde f_n$, if $p^{a_n}$ and $p^{b_n}$ kills the torsion part of $ \tilde M^*_\inv (f_{ij_n} (L_i))/ \tilde g_n (M_i)$ and $M_{j_n} / \tilde \alpha_n (\tilde M^*_\inv (f_{ij_n} (L_i)))$ respectively, then $p^{a_n +b_n}$ kills the torsion part of  $M_{j_n}/ g_{ij_n} (M_i)$. Now we claim that there exists an integer $m_i$ such that $p^{m_i}$ kills the torsion part of $ L_{j_n}/ f_{ij_n} (L_i)$ for all $n$. Let us first accept the claim. Then Lemma \ref{lem: property of inv} (4) proves  that
$p ^{m_i}$ kills the torsion part of $M_{j_n} / \tilde \alpha_n (\tilde M^*_\inv (f_{ij_n} (L_i)))$. Since the exact sequences
$$0 \to \t{Ker}(f_{ij_n}) \to L_i  \overset {\tilde f_n }{\to}  f_{ij_n} (L_i) \to 0 $$
are isomorphic for all $n$, we see that $a_n$ is independent on $n$. Hence there exists an $m_i'$ independent on $n$ that $p^{m'_i}$ kills the torsion part of $M_{j_n}/g_{ij_n} (M_i)$, and this contradicts the existence of $y_n$. Hence $g_i (y)$ has to be $0$ and $M$ is $p$-adically separated.

Now it suffices to settle the claim. Since all $f_{ij_n}(L_i)$ are isomorphic, without loss of generality, we may assume that $f_{ij_n}$ are  injective for all $n$. Hence we may regard $L_i$ is a submodule of $L_{j_n}$ via $f_{ij_n}$. Let $T_n  = (\Q_p \otimes_{\Z_p} L_i)  \cap L_{j_n}$. Then $T_n$ are increasing finite free $\Z_p$-lattices inside $\Q_p \otimes_{\Z_p} L_i$. It suffices to show that $ T_m = T_{m+1}$ if $m$ is sufficiently large. In fact, if $T_n$ is keeping increasing its size then it is easy to show there must be a $x \in L_i$ such that $x = p^l y_{n_l}$ for a $y_{n_l} \in T_{n_l} \subset L_{n_l}$ for $l \geq 1$. But this contradicts that $L$ is $p$-adically separated.


\end{proof}

The above proposition is actually motivated by the situations in \cite{Emerton1} and \cite{Cais}. In the following, we discuss the situation in \S7.2 of \cite{Emerton1} and also use notations there. Fix a compact open subgroup $K^p$ of $\GL_2( \widehat{\Z}^p)$;  We
refer to $K^p$ as the ``tame level". Fix a finite extension $E$ of $\Q_p$ with ring of integers
$\O_E$. Write  $$\t{H}^1 (K^p)_{\O_E} :=  \varinjlim_{K_p} \t{H}^1 ( Y(K^p K_p)/_{\overline \Q}, \O_E), $$ where the inductive limit is taken over all open subgroups $K_p$ of $\GL_2(\Z_p)$, $Y(K^p K_p)$ is the modular curve determined by $K^p K_p$ and the cohomology is \'etale cohomology. By Lemma 7.2.1 in \cite{Emerton1}, $\t{H}^1 (K^p) _{\O_E}$ is torsion free and $p$-adically separated. Set $\widehat{\t{H}} ^1 (K^p) _{\O_E}$ the $p$-adic completion of $ \t{H} ^1 ( K^p)_{\O_E}$ and $\widehat{\t{H}} ^1 (K^p) _{E} := E \otimes_{\O_E} \widehat{\t{H}} ^1 (K^p) _{\O_E}$. Lemma 7.2.5 \emph{loc. cit.} showed that $\widehat{\t{H}} ^1 ( K^p)_{E}$ is an admissible unitary representation of $\GL_2 (\Q_p)$.

Since we only concerns the local properties at $p$, we restrict all  the above Galois modules (they are $\Z_p[\gal (\overline \Q/ \Q)]$-modules) to $\gal (\overline{\Q_p}/\Q_p )$ but still use the same notations. Now apply the functor $\tilde M^*_\inv $ to $ \t{H}^1 ( Y(K^p K_p)/_{\overline \Q}, \O_E)$ and set $M_{K_pK^p} := \tilde M^*_\inv (\t{H}^1 ( Y(K^p K_p)/_{\overline \Q}, \O_E)). $ By comparison theorem, $M_{K_pK^p}$ is obviously a $\Z_p$-lattice in the de Rham cohomology $\t{H}^1_\dR ( Y(K^p K_p)_{\overline{\Q_p}}, E).  $ Proposition \ref{prop: directlimit} implies that $ M_{K^p} : = \varinjlim_{K_p} M_{ K_p K^p}$ is $p$-adic separated. Define $ \widehat{\t{H}}^1_\dR ( K^p) $ the $p$-adic completion of $M_{K^p}$ and $  \widehat{\t{H}}^1_\dR ( K^p)_E := E \otimes_{\O_E}  \widehat{\t{H}}^1_\dR ( K^p)  $. By the construction of $\tilde M_\inv$, we easily see that $\widehat{\t{H}}^1_\dR ( K^p)_E $ has a continuous action of $\GL_2 (\Q_p)$. It is natural to ask the following questions:
\begin{question}\label{Q} \begin{enumerate} \item What we can say on the $\GL_2 (\Q_p)$-action on $ \widehat{\t{H}}^1_\dR ( K^p)_E$? Is it an admissible unitary representation of $\GL_2 (\Q_p)$?
\item Is there any relation between $\widehat{\t{H}} ^1 ( K^p)_{E}$  and $\widehat{\t{H}} ^1_\dR ( K^p)_{E}$?  Could we build a comparison theorem to compare them?
\end{enumerate}
\end{question}

Let $T$ be a $G_K$-stable $\Z_p$-lattice inside a De Rham representation. We may define $M^{W(\bar k)}_\inv (T) := \varinjlim_{F}M_{\inv, F}(T^\v)$ where $F$ runs through all finite extensions of $K$. It is easy to see that $M^{W(\bar k)}_\inv(T)$ is a finite free $W(\bar k)$-module with a $(\varphi, N, G_K)$-action and the $G_K$-action factors through a  finite quotient of $G_K$.  We note that Proposition \ref{prop: directlimit} is still valid after replacing $\tilde M_\inv$ by $M^{W(\bar k)}_\inv$,  because the proof only uses the Lemma \ref{lem: property of inv} (3) and (4) and it is easy to check that these are still valid for $M^{W(\bar k)}_\inv$. If we apply the functor $M^{W(\bar k)}_\inv$ to $\t{H}^1 (K^p)_{\O_E}$ the above then the direct limit of $M^{W(\bar k)} _ \inv ( \t{H}^1 ( Y (K^p K_p)/ {\overline \Q_p}, \O_E))$ is separated and can be completed. Denote this completion $\widehat{\t{H}}^{1, W(\bar k)}_\st (K^p)_{\O_E}$, which has a natural $(\varphi, N , G_K)$-action and a $\GL_2 (\Q_p)$-action. It is natural  ask Question \ref{Q} (1) for $\widehat{\t{H}}^{1, W(\bar k)}_\st (K^p)_{\O_E}$ again and the question that what are relations between $\widehat{\t{H}}^{1, W(\bar k)}_\st (K^p)_{\O_E}$, $\widehat{\t{H}}^1 (K^p) _E $ and $\widehat{\t{H}}^1 _\dR (K^p)_E$.

Finally, we may define  $M^{W(\bar k)}_\st  (T) := \varinjlim_{F}M_{\st, F}(T^\v)$, which is another $W(\bar k)$-lattice inside $\Q_p \otimes _{\Z_p}M^{W(\bar k)} _{\inv} (T)$ which is $(\varphi, N , G_K)$-stable. Though the functor $M_\st^{W(\bar k)}$ enjoys many good properties (for example, $ M_{\st, F} (\t{H}^1 (Y(K^p K_p)/ {\overline \Q_p}, \O_E))$ does has geometric interpretation  if $Y(K^p K_p)$ has a good reduction over $F$), we do not know the direct limit of $M^{W(\bar k)}_ \st (\t{H}^1 (Y(K^p K_p)/ {\overline \Q_p}, \O_E)$ is $p$-adic separated as the functor $M_\st$ in general is not left exact. Hence its $p$-adic completion may contain few information to understand $\widehat{\t{H}}^1 (K^p)_E$.

\section{Errata for \cite{liu-wd}}

Theorem 2.1.3 in \cite{liu-wd} claimed that the functor $M_\st$ is left exact. Unfortunately this is false as Example \ref{ex: 2} explains. Given  an exact sequence of lattices in semi-stable representations $0 \to T' \to T \to T'' \to 0$, Lemma 2.5.4 in \cite{liu-wd} showed  that the associated sequence of Kisin module  $0 \to \M'' \to \M \to \M' \to 0$ is left exact. But it is not true in general that the sequence $$ 0 \to \M''/u\M'' \to \M/ u \M  \to \M'/ u \M ' \to 0  $$ is exact on $\M/ u\M$. This is  exactly the mistake (of the sentence right before Lemma 2.5.4 in \cite{liu-wd}) when proceeded the proof that $M_\st$ is left exact.

Theorem 2.1.3 in \cite{liu-wd} except this claim is still correct and we have not used this claim in \cite{liu-wd} and our other papers.
\bibliographystyle{amsalpha}
\bibliography{BIBLIO1}


\end{document}